\documentclass[11pt,reqno]{amsart}
\usepackage{amsmath, amsthm, amssymb}
\usepackage{amsmath,amscd}
\usepackage{mathabx}
\usepackage{hyperref}
\usepackage{mathrsfs}
\usepackage{xcolor, changepage} 
\usepackage{graphicx}
\usepackage{titletoc}
\usepackage{enumerate}
\usepackage{accents}
\usepackage{comment}
\usepackage{esint}

\usepackage{tikz}
\usetikzlibrary{matrix,arrows}
\usetikzlibrary{shapes}
\usetikzlibrary{calc}
\usetikzlibrary{arrows}
\usetikzlibrary{decorations.pathreplacing,decorations.markings}
\usepackage[all]{xy}
\usepackage{tikz-cd}
\usepackage[a4paper,margin=30mm]{geometry}
\usepackage{caption}
\usepackage{subcaption}

\theoremstyle{plain}
\newtheorem{theorem}{Theorem}
\newtheorem{proposition}[theorem]{Proposition} 
\newtheorem{lemma}[theorem]{Lemma}
\newtheorem{remark}[theorem]{Remark}
\newtheorem{corollary}[theorem]{Corollary}

\newtheorem{example}[theorem]{Example}

\numberwithin{theorem}{section}

\theoremstyle{definition}
\newtheorem{definition}[theorem]{Definition}

\DeclareMathOperator{\Lip}{Lip}
\DeclareMathOperator{\dist}{dist}
\DeclareMathOperator{\diam}{diam}

\DeclareMathOperator{\degg}{deg}

\DeclareMathOperator{\divv}{div}
\DeclareMathOperator{\spt}{spt}

\DeclareMathOperator{\Ric}{Ric}
\DeclareMathOperator{\Capp}{Cap}

\numberwithin{equation}{section}

\begin{document}
\title{A proof for the Riemannian positive mass theorem up to dimension 19}
\author{Yuchen Bi}
\address[Yuchen Bi]{Mathematical Institute, Department of Pure Mathematics, University of Freiburg, Ernst-Zermelo-Stra{\ss}e 1, D-79104 Freiburg im Breisgau, Germany}
\email{yuchen.bi@math.uni-freiburg.de}

\author{Tianze Hao}
\address[Tianze Hao]{Beijing International Center for Mathematical Research, Peking University, Beijing, 100871, P. R. China
}
\email{haotz@pku.edu.cn}

\author{Shihang He}
\address[Shihang He]{Key Laboratory of Pure and Applied Mathematics,
School of Mathematical Sciences, Peking University, Beijing, 100871, P. R. China
}
\email{hsh0119@pku.edu.cn}
\author{Yuguang Shi}
\address[Yuguang Shi]{Key Laboratory of Pure and Applied Mathematics,
School of Mathematical Sciences, Peking University, Beijing, 100871, P. R. China
}
\email{ygshi@math.pku.edu.cn}

\author{Jintian Zhu}
\address[Jintian Zhu]{Institute for Theoretical Sciences, Westlake University, 600 Dunyu Road, 310030, Hangzhou, Zhejiang, People's Republic of China}
\email{zhujintian@westlake.edu.cn}

\thanks{ T. Hao is supported by the Postdoctoral Fellowship Program of CPSF under Grant Number GZB20250708 and the Boya Postdoctoral Fellowship, Peking University. S. He is partially supported by NSFC Grant No.125B2005. Y. Shi is funded by NSFC12431003. J. Zhu is partially supported by National Key R\&D Program of China 2023YFA1009900, NSFC Grant No.12401072, Zhejiang Provincial Natural Science Foundation of China under Grant No.LQKWL26A0101 and the start-up fund from Westlake University.}

\begin{abstract}
   In this paper, we prove the Riemannian positive mass theorem up to dimension $19$, building on a combination of torical symmetrization and the singularity blow-up technique developed in \cite{HSY26}, together with the generic regularity theory for area-minimizing hypersurfaces established in \cite{CMS23, CMSW2025}.
Similar ideas are also employed to investigate the Geroch conjecture up to dimension $12$.
\end{abstract}

\maketitle
\tableofcontents

\section{Introduction}
\subsection{Positive mass theorem}
The study of the Positive Mass Theorem (PMT)
forms a cornerstone of modern Riemannian geometry and geometric analysis, originating at the interface of mathematical general relativity and global differential geometry.  Motivated by the physical principle that isolated gravitational systems should have nonnegative total energy, the three-dimensional PMT was first rigorously established by Schoen and Yau \cite{SY79} using minimal hypersurface methods. An elegant spinorial proof was later given by Witten \cite{Wi1981}. More recently, alternative proofs for the three-dimensional PMT have been also established by Bray, Kazaras, Khuri, and Stern \cite{BKKS22} using level-set techniques, and independently by Agostiniani, Mazzieri, and Oronzio \cite{AMO24} employing a potential-theoretic approach. The level-set method has further led to a resolution of the stability conjecture for PMT, see \cite{Huisken-Ilmanen, Dong-Song} for details.

The higher-dimensional cases are also of significant importance, particularly concerning the new challenges that emerge. Witten's spinorial proof works for all spin asymptotically flat (AF) manifolds, regardless of dimension, but it does not apply to the non-spin case. On the other hand, minimal hypersurfaces possibly develop singularities in higher dimensions; therefore, the dimension-descent argument with smooth minimal hypersurfaces works only up to dimension seven \cite{SY81, Schoen1989}. (Note that this range can be extended to dimension eleven with the help of generic regularity theory developed in a series works \cite{Smale93, CMS24, CMSW2025}.) To overcome this difficulty, Schoen and Yau \cite{SY22} later proposed a dimension-descent scheme using minimal slicings with singularity to extend the PMT to both non-spin manifolds and higher dimensions. 

Recently, PMT has been further extended to AF manifolds with arbitrary ends (see Definition \ref{def: AF manifold}) in subsequent works \cite{LUY21, LLU23, Zhu23, Zhu24, BC05, CZ24}. These results have led to a variety of applications, including the establishment of Liouville theorems in conformal geometry \cite{SY88} and techniques for singularity desingularization \cite{DaSW2024, HSY26}\footnote{An earlier version of \cite{HSY26} was posted as \cite{HeSY2025}. The preprint \cite{HeSY2025} was later split into \cite{HSY26} and another paper for improved readability}.

In this paper, we combine Schoen and Yau's dimension-descent idea with the singularity desingularization technique developed by Yu, He, and Shi (the third- and fourth-named authors) to establish the PMT for AF manifolds with arbitrary ends up to dimension 19.

Throughout this paper, we always consider AF manifolds  {\it with arbitrary ends} in the following sense.
\begin{definition}\label{def: AF manifold}
    A triple $(M^n,g,E)$ is said to be AF  with arbitrary ends if
    \begin{itemize}
        \item $(M^n,g)$ is a complete Riemannian manifold with $n:=\dim M\geq 3$;
        \item $E$ is a component of $M-K$ for some compact subset $K$ of $M$;
        \item $E$ is AF in the sense that $E$  is diffeomorphic to $\mathbf R^n\setminus B_1$, and under the Cartesian  coordinates on $E$, the components of metric $g$ in $E$ satisfies
        $$|g_{ij}-\delta_{ij}|+|x||\partial g_{ij}|+|x|^2|\partial^2 g_{ij}|=O(|x|^{-\tau})\mbox{ for some }\tau>\frac{n-2}{2}$$
        and the scalar curvature
        $$R_g\in L^1(E,g).$$
    \end{itemize}
\end{definition}

We also recall the following definition of mass.

\begin{definition}
    Let $(M^n,g,E)$ be AF with arbitrary ends. Then the mass of $(M^n,g,E)$ is defined by
    $$m(M,g,E)=\frac{1}{2\omega_{n-1}}\lim_{\rho\to+\infty}\int_{\mathbf S^{n-1}(\rho)}(\partial_ig_{ij}-\partial_j g_{ii})\nu^j\,\mathrm d\sigma_x,$$
    where $\omega_{n-1}$ denotes the volume of the unit round $(n-1)$-sphere $\mathbf S^{n-1}$, $\mathbf S^{n-1}(\rho)$ denotes the coordinate sphere $\{|x|=\rho\}$, $\nu=(\nu^j)$ denotes the Euclidean unit normal of $\mathbf S^{n-1}(\rho)$ pointing to the infinity of $E$, and $\mathrm d\sigma_x$ denotes the Euclidean area element of $\mathbf S^{n-1}(\rho)$.
\end{definition}

Our main theorem is stated as follows.
\begin{theorem}\label{Thm: PMT}
   Let  $(M^n,g,E)$ be an AF manifold with  arbitrary ends and nonnegative scalar curvature, where $n:=\dim M\leq 19$. Then we have $m(M,g,E)\geq 0$ and the equality holds if and only if $(M,g)$ is isometric to the Euclidean $n$-space.
\end{theorem}

The proof builds on Schoen and Yau's dimension-descent argument \cite{SY81, Schoen1989}, which is effective up to dimension seven. In higher dimensions, area-minimizing hypersurfaces may develop singularity, and iterating the descent causes later minimizers to interact with the existing singularities—posing the central challenge in the proof of higher dimensional PMT. Unlike the approach in \cite{SY22}, which preserves singularities throughout the descent, we develop the singularity desingularization method to resolve them following the idea of the previous work \cite{HSY26}. Actually, we are able to establish the following new blow-up result, where only the Hausdorff dimension of the singular set is required compared to previous results.

\begin{proposition}\label{Prop: desingularization}
Use the notation and terminology from Section \ref{Sec: blow-up}.
      Let $(X,d,\mu)$ be a non-parabolic almost-manifold metric-measure space with Ahlfors dimension $m$ satisfying local Poincar\'e inequality. Assume
      \begin{itemize}
          \item either the singular set $\mathcal S$ satisfies $\dim_{\mathcal H}(\mathcal S)<\frac{m-2}{2}$;
          \item or the singular set $\mathcal S$ satisfies $\mathcal H^{\frac{m-2}{2}}(\mathcal S\cap U)<+\infty$ for any bounded subset $U$ of $X$ and $(X,d)$ can be bi-Lipschitzly embedded into a smooth Riemannian manifold.
      \end{itemize}
       Then there is a smooth harmonic function $G_{\mathcal S}$ on the regular set $\mathcal R$ such that the conformal metric
    $$\tilde g:=(1+\delta G_{\mathcal S})^{\frac{4}{m-2}}g$$
    is complete on $\mathcal R$ for any $\delta>0$.
\end{proposition}

The conformal blow-up trick originated in Schoen's resolution of the Yamabe problem \cite{Schoen84} and has been applied to attack Schoen's conjecture by various authors \cite{LM19, Kaz24, WX2024, DaSW24b}, as well as to study the PMT \cite{DaSW2024, HSY26}. See also \cite{CLZ24b} for other related work. Given these developments, we anticipate that our Proposition \ref{Prop: desingularization} will have more applications in other geometric problems.

The use of the blow-up makes it natural to work within the framework of AF manifolds with arbitrary ends. However, to transform the positivity of curvature from a spectral sense (arising from the stability inequality) to a pointwise sense, conformal deformation presents a difficulty: it is hard to preserve completeness on those arbitrary ends. As an alternative approach, we will apply the torical symmetrization procedure dating back to \cite{FcS1980, Gro18}, following \cite{HSY26}, and work with the $\mathbf{T}^k$-warped asymptotically locally flat (ALF) manifolds introduced below.

\begin{definition}\label{def: ALF}
    A triple $(M,g,E)$ is said to be $\mathbf T^k$-warped ALF if
    \begin{itemize}
        \item $(M,g)$ is a complete Riemannian manifold such that $M=M_0\times \mathbf T^k$ and 
        $$g=g_0+\sum_{i=1}^ku_i^2\,\mathrm d\theta_i\otimes \mathrm d\theta_i,$$
        where $n=\dim M_0\geq 3$, $g_0$ is a smooth metric on $M_0$, $u_i$ are smooth positive functions on $M_0$, and $\theta_i\in [0,1]$ denotes the arc-length parametrization of $i$-th unit circle component of $\mathbf T^k$;
        \item $E=E_0\times \mathbf T^k$ is an end of $M$ where $E_0$ is a component of $M_0\setminus K$ for some compact subset $K$ of $M_0$;   
        \item $E=E_0\times \mathbf T^k$ is ALF in the sense that $E_0$ is diffeomorphic to $\mathbf R^n\setminus B_{r_0}$ for some constant $r_0>1$, and  under the Cartesian  coordinates on $E_0\times \mathbf T^k$, the components of metric $g$ in $E$ satisfies
       \begin{align}\label{Eq: decay condition}
			|g_{ij}-\bar{g}_{ij}|+|x||\partial g_{ij}|+|x|^2|\partial^2 g_{ij}| = O(|x|^{-\tau})\mbox{ for some }\tau >\frac{n-2}{2},
		\end{align}
		and also $R_g\in L^1(E,g)$, where $x$ denotes the Cartesian coordinate on $\mathbf R^n$ and $\bar g$ denotes the reference metric given by $\bar g=\mathrm dx\otimes \mathrm dx+\mathrm d\theta\otimes \mathrm d\theta$.       
    \end{itemize}
\end{definition}

Let us recall the similar definition of mass for $\mathbf T^k$-warped ALF manifolds.

\begin{definition}\label{def:T^k ALF mass}
    Let $(M,g,E)$ be $\mathbf T^k$-warped ALF.  Then the  mass of $(M,g,E)$ is defined by
        \begin{align*}
            m(M,g,E) = \frac{1}{2\omega_{n-1}}\lim_{\rho\to\infty}\int_{\mathbf S^{n-1}(\rho)\times \mathbf{T}^k}(\partial_ig_{ij}-\partial_j g_{\alpha\alpha})\bar\nu^j\,\mathrm d\bar\sigma,
        \end{align*}
        where $\omega_{n-1}$ denotes the volume of the unit round $(n-1)$-sphere $\mathbf S^{n-1}$, $\mathbf S^{n-1}(\rho)$ is the coordinate sphere $\{|x|=\rho\}$ in $\mathbf R^n$, $\bar\nu=(\bar\nu^j,0)$ is the unit normal of $\mathbf S^{n-1}(\rho)\times \mathbf T^k$ pointing to the infinity of end $E$, $\mathrm d\bar\sigma$ denotes the area element of $\mathbf S^{n-1}(\rho)\times \mathbf T^k$ with respect to the reference metric $\bar g$, the symbols $i,j$ run over the index of $\mathbf{R}^n$, and the symbol $\alpha$ runs over all index of $\mathbf R^n\times \mathbf T^k$.
  \end{definition}

Within the framework of $\mathbf{T}^k$-warped ALF manifolds, the dimension-descent argument is mirrored by the following torical symmetrization procedure.

\begin{proposition}\label{Prop: torical symmetrization}
      
Suppose $(M,g,E)$ is a $\mathbf T^k$-warped ALF manifold with negative mass and nonnegative scalar curvature. If $k+3< \dim M\leq 19$, then we can construct a new $\mathbf T^{k+1}$-warped ALF manifold $(\widehat M,\widehat g,\widehat E)$ with negative mass and nonnegative scalar curvature, where $\dim \widehat M=\dim M$.
  \end{proposition}
  \begin{remark}
     The dimensional constraint stems from the (generic) regularity theory for area-minimizing hypersurfaces in geometric measure theory. In broad terms, we establish that if such hypersurfaces possess singularities with locally finite $(\dim M-\mathfrak d)$-dimensional Hausdorff measure for a generic metric, then the torical symmetrization procedure remains valid up to dimension $2\mathfrak d - 3$. More precisely, this conclusion follows from the dimension inequality
       $$\dim M-\mathfrak d\leq\frac{\dim M-3}{2}\mbox{ when }\dim M\leq 2\mathfrak d-3$$
       in conjunction with Proposition \ref{Prop: desingularization}.
  \end{remark}

\subsection{The Geroch conjecture}
The Geroch conjecture asserts that the torus $\mathbf{T}^n$ does not admit any smooth Riemannian metric with positive scalar curvature. This was verified by Schoen and Yau \cite{SY79b} for $3 \leq n \leq 7$, and later \cite{SY22} for all dimensions using minimal hypersurfaces, and by Gromov and Lawson \cite{GL80} for all dimensions using spinors. See also \cite{ CL2024, LUY24, WZ22} for further developments and applications. 

The PMT for AF manifolds is closely connected to the Geroch conjecture and its variants: if $\mathbf{T}^n\# X$ admits no complete metric of positive scalar curvature for any closed $n$-dimensional manifold $X$, then the PMT holds in dimension $n$. This suggests that our method might be applicable to establishing positive scalar curvature obstructions on torical connected sums. In practice, however, the situation is more delicate because the conformal blow-up procedure can disrupt the topological descent structure. Nevertheless, by combining the generic regularity theory with the theory of open Schoen-Yau-Schick (SYS) manifolds introduced in \cite{SWWZ24}, we provide a proof of  such obstructions for dimension up to $12$.  The general case is treated in  \cite{SY22}.

\begin{theorem}\label{thm:12 dim Geroch}
    Let $X^n$ be any closed manifold and $n\leq 12$. Then $\mathbf{T}^n\# X^n$ admits no  complete Riemannian metric with positive scalar curvature. Moreover, if  $g$ is any complete Riemannian metric on $\mathbf{T}^n\# X^n$ with  nonnegative scalar curvature, then $(\mathbf{T}^n\# X^n,g)$ is isometric to a flat torus $\mathbf{T}^n$.
        \end{theorem}

\subsection{Sketch of the proofs}
The central task in Proposition \ref{Prop: desingularization} is to construct a positive singular harmonic function with a sufficient blow-up rate near $\mathcal{S}$. To this end, our approach integrates positive Green's functions along $\mathcal{S}$ against a Radon measure $\nu$. By local Ahlfors $m$-regularity, these Green's functions admit a lower bound of order $r^{2-m}$. A key novelty, inspired by \cite{Karakhanyan24}, is that completeness after conformal deformation can be ensured by choosing $\nu$ so that a certain Wolff potential diverges along $\mathcal{S}$. Lemma \ref{lem:wolff-divergent-metric} and Lemma \ref{Lem: Wolff divergent} provide the construction of such a measure, thereby completing the proof.

The proof of Proposition \ref{Prop: torical symmetrization} follows the standard arguments in \cite{Schoen1989}, with modifications incorporating Proposition \ref{Prop: desingularization}. The construction proceeds as follows. First, we deform the $\mathbf{T}^k$-warped ALF manifold $(M,g)$ so that the metric becomes conformally flat near infinity, the scalar curvature increases in the interior, and the negative mass is preserved. After establishing a generic regularity result for certain area-minimizing hypersurfaces (see Corollary \ref{Cor:generic regularity}), we construct a strongly stable minimal hypersurface, possibly with a singular set of Hausdorff dimension less than $(\dim M-3)/2$. This hypersurface is shown to be $\mathbf{T}^k$-warped ALF near infinity and to have zero mass. We then apply Proposition \ref{Prop: desingularization} to this minimal hypersurface, blowing up the singularity so that it becomes invisible at infinity. The strong stability implies that the minimal hypersurface has positive scalar curvature in a spectral sense. Using a warping construction originating from \cite{FcS1980}, we obtain a new $\mathbf{T}^{k+1}$-warped ALF manifold with nonnegative scalar curvature and negative mass.

To complete the proof of Theorem \ref{Thm: PMT}, we observe that the PMT for $\mathbf{T}^k$-warped ALF manifolds holds in all dimensions provided there is sufficient $\mathbf{T}^k$-symmetry (see Proposition \ref{Prop: PMT ALF}). Suppose, for contradiction, that the original AF manifold with arbitrary ends has negative mass. Applying Proposition \ref{Prop: torical symmetrization} inductively, we eventually obtain a $\mathbf{T}^k$-warped ALF manifold with nonnegative scalar curvature and sufficient $\mathbf{T}^k$-symmetry, yet still with negative mass—a contradiction.

To establish Theorem \ref{thm:12 dim Geroch}, we proceed as follows. By the generic regularity theory of \cite{CMSW2025}, we take an area-minimizing hypersurface $\Sigma \subset \mathbf{T}^n \# X$ in the homology class represented by a standard coordinate subtorus $\mathbf{T}^{n-1} \subset \mathbf{T}^n$, whose singular set $\mathcal S$ satisfies $\dim_{\mathcal H}(\mathcal S) < 1$. This dimensional estimate enables us to verify that $\Sigma \setminus \mathcal S$ is an open SYS manifold in the sense of \cite{SWWZ24}. Although \cite{SWWZ24} treats the smooth case up to dimension $7$, the same positive scalar curvature obstruction remains valid in dimension $11$ when coupled with the generic regularity input from \cite{CMSW2025}. Finally, we blow up $\mathcal S$ in the spirit of Proposition \ref{Prop: desingularization}, producing a complete conformal metric on $\Sigma \setminus \mathcal S$ and thereby contradicting this obstruction.
  
\subsection{Organization of The Paper}
 The remainder of this paper is organized as follows.

In Section 2, we introduce the framework of almost-manifold metric measure spaces and construct the desired positive singular harmonic function that blows up along the singular set $\mathcal S$, thereby establishing Proposition \ref{Prop: desingularization}.

In Section 3, we illustrate the principal steps toward the torical symmetrization procedure, namely Proposition \ref{Prop: torical symmetrization}. These steps consist of a metric deformation for asymptotic modification (see Proposition \ref{prop: density}), the construction of a strongly stable area-minimizing hypersurface (see Proposition \ref{Prop: minimizing hypersurface}), and the desingularization argument based on Proposition \ref{Prop: desingularization}.

In Section 4, we establish a generic regularity result up to dimension $11$ for the aforementioned area-minimizing hypersurfaces (see Proposition \ref{Prop: generic regularity}). By inputting this generic regularity result into the arguments in Section 2, we complete the proof of Proposition \ref{Prop: torical symmetrization} and then provide the final proof of Theorem \ref{Thm: PMT}.

In Section 5, we adapt similar ideas to prove Theorem \ref{thm:12 dim Geroch}.

\bigskip
\section*{Acknowledgements}
We thank Zhihan Wang for pointing out that the generic regularity result in \cite{CMSW2025} is purely local, and Tongrui Wang for helpful conversations on the stability of $G$-invariant area-minimizing hypersurfaces. We are also grateful to Professor Shing-Tung Yau for his interest in this work. The authors also express their gratitude to Professor Gang Tian for his constant support and encouragement.

\section{Singularity blow-up for almost manifolds}\label{Sec: blow-up}
\subsection{The set-up}

 In this section, we work with a metric measure space $(X, d,\mu)$, where $(X,d)$ is complete and $\mu$ is a Borel regular measure satisfying $\mu(B_r(x))<+\infty$ for any $x\in X$ and $r>0$. Due to independent interest, we would like to do singularity blow-up in the most general form, so we work with {\it almost-manifold} metric measure spaces. Recall that a metric measure space $(X,d,\mu)$ is called almost-manifold if
\begin{itemize}
    \item[(AM)] there is a positive integer $m\geq 3$ such that 
    \begin{itemize}
        \item $(X,d,\mu)$ is {\it locally Ahlfors $m$-regular}: for any bounded subset $U$ of $X$ there are constants $\varepsilon=\varepsilon(U)$ and $C=C(U)$ such that
        $$C^{-1}r^m\leq\mu(B_r(x))\leq Cr^{m}$$
        for any $x\in U$ and $r\in (0,\epsilon)$;
        \item $(X,d,\mu)$ admits a {\it regular-singular decomposition} $$X=\mathcal R\sqcup\mathcal S,$$ where $\mathcal R$ is a {\it connected} manifold of dimension $m$ equipped with a Riemannian metric $g$ such that $d_L=d_g$ and $\mu=\mu_g$ and $\mathcal S$ is a closed subset of $X$ with $\mathcal H^{m-2}(\mathcal S)=0$, where $d_L$ denotes the length metric associated to $d$, and $d_g$ and $\mu_g$ are metric and volume measure induced from $g$ respectively. Note that we have $\mu(\mathcal S)=0$ in particular.
    \end{itemize}
\end{itemize}

For an almost-manifold metric measure space, we can introduce Lebesgue and Sobolev function spaces without much effort. Let $\Omega$ be any open subset of $X$.
 We shall use $L^p(\Omega)$ to denote the completion of $C_0^\infty(\mathcal R\cap \Omega)$ with respect to the norm
$$\|\phi\|_{L^p(\Omega)}:=\left(\int_{ X}|\phi|^p\,\mathrm d\mu\right)^{\frac{1}{p}}$$
and use $W^{1,2}_0(\Omega)$ to denote the completion of $C_0^\infty(\mathcal R\cap\Omega)$ with respect to the norm
$$\|\phi\|_{W_0^{1,2}(\Omega)}:=\left(\int_{\mathcal R}|\phi|^2+|\nabla_g\phi|^2\,\mathrm d\mu\right)^{\frac{1}{2}}.$$
We will also use the classical Sobolev spaces on the Riemannian manifold $(\mathcal R,g)$.

To construct certain harmonic function we also require the metric measure space $(X,d,\mu)$ to satisfy:
\begin{itemize}
    \item[(LPI)] {\it Local $(1,2)$-Poincar\'e Inequality}: for any bounded subset $U$ of $X$ there are constants $\varepsilon=\varepsilon(U)>0$,  $\lambda=\lambda(U)\geq 1$ and $C_P=C_P(U)>0$ such that we have
    $$\fint_{B}|\phi-\phi_B|\,\mathrm d\mu\leq C_Pr\left(\fint_{\mathcal R\cap\lambda B}|\nabla_g \phi|^2\,\mathrm d\mu\right)^{\frac{1}{2}}\mbox{ for any }\phi\in W^{1,2}_0(X),$$
    where $B$ and $\lambda B$ denote $B_r(x)$ and $B_{\lambda r}(x)$ for any $x\in U$ and $r\in (0,\varepsilon)$ respectively, and $\phi_B$ denotes the average value of $\phi$ in $B$.
\end{itemize}
We point out that the local Ahlfors $m$-regularity in (AM) actually yields the local volume doubling property. That is, we have
\begin{itemize}
    \item[(LVD)] {\it Local Volume Doubling}: for any bounded subset $U$ of $X$ there are constants $\varepsilon=\varepsilon(U)>0$ and $C_D=C_D(U)>0$ such that
    $$\mu(B_{2r}(x))\leq C_D\mu(B_r(x))$$
    for any $x\in U$ and $r\in (0,\varepsilon)$.
\end{itemize}

\begin{example}
    Area-minimizing hypersurfaces and complete Riemannian manifold with $L^\infty(M)\cap C^\infty(M-\mathcal S)$ Riemannian metric are typical examples of metric measure spaces satisfying (AM) and (LPI).
\end{example}

\subsubsection{Preliminary results}

In the following, we always use $U$ to denote a {\it bounded} open subset of $X$. First let us collect preliminary lemmas for analysis.
\begin{lemma}\label{Lem: strong poincare}
  There are constants $\varepsilon=\varepsilon(U)>0$, $\lambda=\lambda(U)\geq 1$ and $C=C(U)>0$ such that  we have 
    $$\left(\fint_{B}|\phi-\phi_B|^{\frac{2m}{m-2}}\,\mathrm d\mu\right)^{\frac{m-2}{2m}}\leq Cr\left(\fint_{\mathcal R\cap\lambda B}|\nabla_g\phi|^2\,\mathrm d\mu\right)^{\frac{1}{2}}\mbox{ for all }\phi\in W^{1,2}_0(X),$$
     where $B$ and $\lambda B$ denote $B_r(x)$ and $B_{\lambda r}(x)$ for any $x\in U$ and $r\in (0,\varepsilon)$ respectively, and $\phi_B$ denotes the average value of $\phi$ in $B$.
\end{lemma}
\begin{proof}
    This follows from the same argument of \cite[Theorem 7.1.13]{Petersen16}, where we work with the $(1,2)$-Poincar\'e inequality instead of the $(1,1)$-Poincar\'e inequality and the relative volume comparison
    $$C^{-1}\left(\frac{R}{r}\right)^m\leq \frac{\mu(B_R(x))}{\mu(B_r(x))}\leq C\left(\frac{R}{r}\right)^m$$
    from the local Ahlfors $m$-regularity in (AM).
\end{proof}

The next Sobolev inequality needs the following characterization for compact subsets.
\begin{lemma}\label{Lem: bounded to compact}
    Bounded closed subsets of $X$ are compact.
\end{lemma}
\begin{proof}
    Using the fact $\mu(U)<+\infty$ for all bounded subsets $U$ and the local Ahlfors $m$-regularity, we see that any bounded closed subset is totally bounded. This further implies compactness since $(X,d)$ is complete.
\end{proof}

In the following, we will assume that $\mathcal R$ is {\it unbounded}.

\begin{lemma}\label{Lem: Sobolev}
    We have the Sobolev inequality
    $$\left(\int_U|\phi|^{\frac{2m}{m-2}}\,\mathrm d\mu\right)^{\frac{m-2}{m}}\leq C_U\int_{\mathcal R\cap U}|\nabla_g\phi|^2\,\mathrm d\mu\mbox{ for all }\phi\in W^{1,2}_0(U).$$
\end{lemma}
\begin{proof}
    This follows from the same argument of \cite[Proposition A.4]{HSY26}. Let us argue by contradiction. If the statement is not true, then we can find a sequence of functions $\phi_i\in W^{1,2}_0(U)$ such that
    $$\int_{U}|\phi_i|^{\frac{2m}{m-2}}\,\mathrm d\mu=1\mbox{ and }\int_{\mathcal R\cap U}|\nabla_g\phi_i|^2\,\mathrm d\mu\to 0\mbox{ as }i\to\infty.$$
   By Lemma \ref{Lem: bounded to compact} we can cover $U$ by finitely many balls $B_l$ where Lemma \ref{Lem: strong poincare} holds for all $B_l$. Up to a subsequence the functions $\phi_i$ converge to a constant $c_l$ in $L^{\frac{2m}{m-2}}(B_l)$. Recall that $\phi_i$ is supported in $U$ which is bounded. From the connectedness and unboundedness of $\mathcal R$ we conclude $c_l=0$ for all $B_l$. In particular, we have 
    $$\int_{U}|\phi_i|^{\frac{2m}{m-2}}\,\mathrm d\mu\to 0\mbox{ as }i\to \infty,$$
    which leads to a contradiction.
\end{proof}

\begin{lemma}\label{Lem: compactness embedding}
    The inclusion $W^{1,2}_0(U)\to L^2(U)$ is compact.
\end{lemma}
\begin{proof}
    It follows from \cite[Theorem 8.1]{HK00} and Lemma \ref{Lem: Sobolev}.
\end{proof}

\begin{lemma}\label{Lem: possion equation}
    Given any Lipschitz function $v$ on $X$ and any function $f\in L^2(U)$ and any vector-valued function $\mathbf F\in L^2(U)$, there is a function $u$ on $X$ such that $-\Delta_g u=f+\divv_g \mathbf F$ in $\mathcal R\cap U$ and $u-v\in W^{1,2}_0(U)$.
\end{lemma}
\begin{proof}
    This follows from a standard variation argument based on the Sobolev inequality from Lemma \ref{Lem: Sobolev}. See \cite[Theorem 8.3]{GT1983} for instance.
\end{proof}

\begin{lemma}\label{Lem: zero capacity}
    $\mathcal S$ has zero $2$-capacity. That is, for any open neighborhood $\mathcal N$ of $\mathcal S$ and $\epsilon>0$ we can find a cut-off function $\eta\in C^\infty(\mathcal R)$ such that $\eta\equiv 0$ in a smaller open neighborhood $\mathcal N'\subset \mathcal N$ of $\mathcal S$ and $\eta\equiv 1$ outside $\mathcal N$ and also
    $$\int_{\mathcal R}|\nabla_g\eta|^2\,\mathrm d\mu\leq \varepsilon.$$
\end{lemma}
\begin{proof}
    See \cite[Lemma 3.2]{HSY26} for instance.
\end{proof}
 
\begin{corollary}\label{Cor: W012 is W12}
    Let $\zeta$ be a Lipschitz function supported in $U$ and $u$ be a function in $W^{1,2}(\mathcal R\cap U)$.
    Then we have $\zeta u\in W^{1,2}_0(U)$.
\end{corollary}
\begin{proof}
Since $u$ is well approximated by $u_i:=\max\{\min\{u,i\},-i\}$ with respect to the $W^{1,2}$-norm. We can assume $u\in L^\infty(U)$ without loss of generality. In this case, we have $$w:=\zeta u\in W^{1,2}(\mathcal R\cap U)\cap L^\infty(U).$$ 
In the following, we show that $w$ can be approximated by functions in $C_0^\infty(\mathcal R\cap U)$. From Lemma \ref{Lem: zero capacity} we can find functions $\eta_i$ such that $$\bigcap_i\{\eta_i\neq 1\}=\mathcal S$$ and
    $$\int_{\mathcal R}|\nabla_g\eta_i|^2\,\mathrm d\mu\to 0\mbox{ as }i\to \infty.$$
    As $i\to\infty$, we can compute
    $$\int_{\mathcal R\cap U}|w-\eta_iw|^2\,\mathrm d\mu\leq \int_{\{\eta_i\neq 1\}}|w|^2\,\mathrm d\mu\to 0$$
    and
    $$\int_{\mathcal R\cap U}|\nabla_g(w-\eta_iw)|^2\,\mathrm d\mu\leq 2\int_{\mathcal R\cap U}|\nabla_g\eta_i|^2w^2+|\nabla_gw|^2\eta_i^2\,\mathrm d\mu\to 0.$$
    To sum up, $w$ is well approximated by functions $\eta_iw$ with respect to the $W^{1,2}$-norm. Note that $\eta_i w$ is supported on a bounded closed subset $K=\spt \zeta\cap \spt \eta_i$ of $\mathcal R$. By Lemma \ref{Lem: bounded to compact} $K$ is compact and so $\eta_iw$ can be approximated by functions in $C^\infty_0(\mathcal R\cap U)$ with respect to $W^{1,2}$-norm. Therefore, we have $w\in W^{1,2}_0(U)$.
\end{proof}

Next let us collect some topological results on connectedness.

\begin{lemma}\label{Lem: connectedness}
    There are constants $\varepsilon=\varepsilon(U)>0$, $\lambda=\lambda(U)>1$ and $C=C(U)>0$ such that for any $x\in U$ and $r\in (0,\varepsilon)$ we have
    \begin{itemize}
        \item $B_r(x)$ is $(C,\lambda)$-quasiconvex in the sense that for any pair of points $x_1$ and $x_2$ in $B_r(x)$ there is a curve $\gamma$ in $B_{\lambda r}(x)$ connecting $x_1$ and $x_2$ with length no greater than $Cd(x_1,x_2)$;
        \item $A_r(x):=B_{2r}(x)-\overline{B_{r/2}(x)}$ is $(C,\lambda)$-quasiconvex in the sense that for any pair of points $x_1$ and $x_2$ in $B_r(x)$ there is a curve $\gamma$ in $B_{2\lambda r}(x)-\overline{B_{x/(2\lambda)}(x)}$ connecting $x_1$ and $x_2$ with length no greater than $Cd(x_1,x_2)$.
    \end{itemize}
\end{lemma}
\begin{proof}
    This is a direct consequence of (LPI) and (AM). See \cite[Theorem 8.3.2 and Theorem 9.4.1]{HKST} for instance.
\end{proof}

\begin{lemma}\label{Lem: local connected regular part}
    There are constants $\varepsilon=\varepsilon(U)>0$ and $\lambda=\lambda(U)\geq 1$ such that $\mathcal R\cap B_r(x)$ is contained in the same component of $\mathcal R\cap B_{\lambda r}(x)$ for any $x\in U$ and $r\in (0,\varepsilon)$.
\end{lemma}
\begin{proof}
    By taking $\varepsilon$ small enough from (LPI) we can guarantee
    \begin{equation}\label{Eq: poincare}
        \fint_{B}|\phi-\phi_B|\,\mathrm d\mu\leq C_Pr\left(\fint_{\mathcal R\cap \lambda B}|\nabla_g\phi|^2\,\mathrm d\mu\right)^{\frac{1}{2}},
    \end{equation}
where $B=B_{r}(x)$ and $\phi\in W^{1,2}_0(X)$. Suppose that the consequence is not true, then we can write $\mathcal R\cap \lambda B=U_1\sqcup U_2$ such that $B\cap U_i\neq\emptyset$ for $i=1,2$. Let us take the function $\phi$ given by $\phi\equiv \mu(U_2\cap B)$ in $U_1$ and $\phi\equiv -\mu(U_1\cap B)$ in $U_2$, and take a Lipschitz cut-off function $\eta$ supported in $\lambda B$ such that $\eta \equiv 1$ on $B$. By Corollary \ref{Cor: W012 is W12} we have $\eta\phi \in W^{1,2}_0(X)$ and so \eqref{Eq: poincare} implies
$$0<\fint_B|\eta\phi|\leq C_P r\left(\fint_{\mathcal R\cap \lambda B}|\nabla_g(\eta\phi)|^2\,\mathrm d\mu\right)^{\frac{1}{2}}=0,$$
which leads to a contradiction.
\end{proof}

Finally, we collect some basic Harnack inequalities.
\begin{lemma}\label{Lem: ball harnack}
  There are constants $\varepsilon=\varepsilon(U)>0$, $\lambda=\lambda(U)> 1$ and $\Lambda=\Lambda(U)>0$ such that if $u\in W^{1,2}(\mathcal R\cap \lambda B)$ is a nonnegative function satisfying $\Delta_g u=0$ in $\mathcal R\cap \lambda B$, then we have
  $$\sup_Bu\leq \Lambda \inf_Bu,$$
  where $B$ and $\lambda B$ denote $B_{r}(x)$ and $B_{\lambda r}(x)$ for any $x\in U$ and $r\in (0,\varepsilon)$.
\end{lemma}
\begin{proof}
    The argument is standard since we have (LPI) and (LVD). See \cite[Section 8]{BB11} for instance. The set $\mathcal S$ is negligible due to Corollary \ref{Cor: W012 is W12}.
\end{proof}
By a ball covering argument we have the following corollary.

\begin{corollary}\label{Cor: annulus harnack}
    There are constants $\varepsilon=\varepsilon(U)>0$, $\lambda=\lambda(U)>1$ and $\Lambda=\Lambda(U)>0$ such that for any $x\in U$ and $r\in (0,\varepsilon)$ if $u$ is a nonnegative function such that $\Delta_g u=0$ in $\mathcal R\cap  \mathring B_{2\lambda r}(x)$ and $u\in W^{1,2}(\mathcal R\cap  B_{2\lambda r}(x)-\overline{B_\rho(x)})$ for any $\rho>0$, then we have
    $$\sup_{A_r(x)} u\leq \Lambda \inf_{A_r(x)} u,$$
    where $A_r(x)$ denotes  $B_{2r}(x)-\overline{B_{r/2}(x)}$.
\end{corollary}
\begin{proof}
  By taking $\varepsilon$ small enough, we know from Lemma \ref{Lem: connectedness} that there are constants $\lambda'$ and $C'$ such that any pair of points $y_1$ and $y_2$ in $A_r(x)$ can be connected by a curve $\gamma$ in $B_{2\lambda' r}(x)-\overline{B_{r/(2\lambda')}}$ with length no greater than $C'r$. From Lemma \ref{Lem: ball harnack} we have
 $$\sup_Bu\leq \Lambda \inf_Bu$$
 for any $B=B_{s}(y)$ with $y\in B_{2\lambda' r}(x)-\overline{B_{r/(2\lambda')}}$ and $s=r/(2\lambda'\lambda'')$, where $\Lambda$ and $\lambda''$ are constants coming from Lemma \ref{Lem: ball harnack}.
    Therefore, we have
    $$u(y_1)\leq \Lambda^{2C'\lambda'\lambda''}\cdot u(y_2)\mbox{ for any }y_1,\,y_2\in A_r(x).$$
    This completes the proof.
\end{proof}

The following version of Harnack inequality will also be used.
\begin{lemma}\label{Lem: enhanced Harnack}
    Let $\mathcal S'$ be a compact subset of $\mathcal S$ and $U$ be a bounded open subset of $X$ such that $\mathcal S'\subset U$ and that $\mathcal R\cap U$ is connected. Given any compact subset $K$ of $U-\mathcal S'$ there are  constants $s>0$ and $\Lambda>1$ such that if $u$ is a nonnegative function such that $\Delta_gu=0$ in $\mathcal R\cap U$ and $u\in W^{1,2}((\mathcal R\cap U)-\overline{B_s(\mathcal S')})$, then we have
    $$\sup_Ku\leq \Lambda \inf_Ku.$$
\end{lemma}
\begin{proof}
   Fix a point $x_0\in\mathcal R\cap U$. First, we claim that for any point $x\in U-\mathcal S'$ we can find constants $s_x>0$, $r_x>0$ and $\Lambda_x>0$ such that if $u$ is a nonnegative function such that $\Delta_gu=0$ in $\mathcal R\cap U$ and $u\in W^{1,2}((\mathcal R\cap U)-\overline{B_{s_x}(\mathcal S')})$, then we have
   $$\Lambda_x^{-1}\cdot u(x_0)\leq \inf_{B_{r_x}(x)}u\leq \sup_{B_{r_x}(x)}u\leq \Lambda_x \cdot u(x_0).$$
   From Lemma \ref{Lem: ball harnack} by setting $r_x$ and $s_x$ small enough we can guarantee
   $$\sup_{B_{r_x}(x)}u\leq C_1\inf_{B_{r_x}(x)}u.$$
   Fix a point $y\in \mathcal R\cap B_{r_x}(x)$. Due to the connectedness of $\mathcal R\cap U$ we can connect $x_0$ and $y$ with a curve $\gamma$ in $\mathcal R\cap U$. Note that we can further decrease the value of $s_x$ such that $\gamma$ is disjoint from $\overline{B_{s_x}(\mathcal S')}$. Then the Harnack inequality from Lemma \ref{Lem: ball harnack} combined with a covering argument gives us
   $$C_2^{-1}u(x_0)\leq u(y)\leq C_2u(x_0)$$
   and so we prove the claim with $\Lambda_x$ depending on $C_1$ and $C_2$. The desired consequence now follows from the fact that $K$ can be covered by finitely many balls $B_{r_x}(x)$.
\end{proof}

\subsection{Singularity blow-up}
\begin{definition}
    A Lipschitz function $v$ on $X$ is superharmonic if $v\geq v_H$ for any harmonic replacement $v_H$ of $v$, where a harmonic replacement $v_H$ of $v$ means a function $v_H$ on $X$ such that $\Delta_gv_H=0$ in $\mathcal R\cap U$ and $v_H-v\in H^1_0(U)$ for some bounded region $U$.
\end{definition}
\begin{definition}
    $(X,d,\mu)$ is called {\it non-parabolic} if there is a non-constant and bounded superharmonic function on $X$.
\end{definition}
We are going to show
\begin{proposition}\label{Prop: weak form}
    Assume that $(X,d,\mu)$ is non-parabolic and that $\mathcal S$ satisfies
    $$\dim_{\mathcal H}(\mathcal S)<\frac{m-2}{2}.$$
    Then there is a smooth positive harmonic function $G_{\mathcal S}$ on $\mathcal R$ such that the conformal metric
    $$\tilde g:=(1+\delta G_{\mathcal S})^{\frac{4}{m-2}}g$$
    is complete on $\mathcal R$ for any $\delta>0$. Moreover, we have
    $$G_{\mathcal S}<C_{E,v}(v-\inf_X v)\mbox{ in }E,$$
    where $E$ is any open subset of $X$ with compact boundary $\partial E$ such that $\overline E\cap \mathcal S=\emptyset$, $v$ is any non-constant bounded superharmonic function on $X$, and $C_{E,v}$ is a positive constant depending on $E$ and $v$.
\end{proposition}

 The proof is based on an exhaustion method, and we need the following lemma.

\begin{lemma}\label{Lem: Ui}
  Given any compact $K$  we can take a bounded open subset $U$ of $X$ such that $K\subset U$ and that $\mathcal R\cap U$ is connected.
\end{lemma}
\begin{proof}
    By Lemma \ref{Lem: local connected regular part} we can take finitely many balls $B_l$ covering $K$, where $\mathcal R\cap B_l$ is contained in a component $\mathcal N_l$ of $\mathcal R\cap \lambda B_l$. Recall that we have $\mu(B_l)>0$ and $\mu(\mathcal S)=0$, so we can take a point $y_l\in \mathcal R\cap B_l$ for each $l$. Recall that $\mathcal R$ is connected, we can find a finite connected graph $G\subset \mathcal R$ such that $y_l$ are the vertex of $G$. For any $y\in G$ we can take a connected ball $B_y\subset\mathcal R$ centered at $y$. Let us take
    $$U=\left(\bigcup_l(B_l\cup \mathcal N_l)\right)\cup\left(\bigcup_{y\in G}B_y\right).$$
    It is easy to verify that $U$ satisfies our requirement.
\end{proof}

Fix a point $o\in \mathcal R$ and denote $$\mathcal S_i=\mathcal S\cap \overline{B_i(o)}.$$
From Lemma \ref{Lem: Ui} we can take a bounded open subset $U_i$ of $X$ such that $\overline{B_i(o)}\subset U_i$ and that $\mathcal R\cap U_i$ is connected.

\begin{lemma}\label{Lem: Greens function at x}
    For any $x\in \mathcal S_i$ we can construct a function $$G_x^{(i)}:X-\{x\}\to [0,+\infty)$$ such that
    \begin{itemize}
        \item $\Delta_g G_x^{(i)}=0$ in $\mathcal R\cap U_i$ and $G_x^{(i)}(o)=1$;
        \item $\eta G_x^{(i)}\in W^{1,2}_0(U_i)$ for any cut-off function $\eta$ vanishing in a neighborhood of $\mathcal S_i$;
        \item there are constants $r_i$ and $c_i>0$, and a positive function $C_i(r)$ such that
        \begin{equation}\label{Eq: lower bound}
            \inf_{\partial B_r(x)}G_x^{(i)}\geq c_ir^{2-m}\mbox{ for any }r\in (0,r_i)
        \end{equation}
        and
        \begin{equation}\label{Eq: upper bound}
            \int_{\mathcal R\cap (U_i-\overline{B_r(\mathcal S_i}))}|\nabla_gG_x^{(i)}|^2\,\mathrm d\mu\leq C_i(r).
        \end{equation}
    \end{itemize}
\end{lemma}
\begin{proof}
    For any small constant $s>0$ we can take a Lipschitz function $v_s$ such that $v_s\equiv 1$ in $B_s(x)$ and $v_s\equiv 0$ outside $B_{2s}(x)$. For simplicity, we will denote $\mathring U_i^s=U_i-\overline{B_s(x)}$. From Lemma \ref{Lem: possion equation} we can construct a function $u_s$ on $X$ such that $\Delta_gu_s=0$ in $\mathcal R\cap \mathring U_i^s$ and also $u_s-v_s\in W^{1,2}_0(\mathring U_i^s)$. Moreover, $u_s$ is a capacity potential, and combined with Lemma \ref{Lem: Sobolev} we have
    \begin{equation}\label{Eq: capacity potential}
        \int_{\mathcal R\cap\mathring U_i^s}|\nabla_gu_s|^2\,\mathrm d\mu=\Capp(\overline{B_s(x)},U_i)>0.
    \end{equation}
    From \eqref{Eq: capacity potential} and the maximum principle we obtain $u_s(o)>0$ and so we can define
    $$w_s:=(u_s(o))^{-1}u_s\mbox{ with }w_s(o)=1.$$ 
    Using the Harnack inequality from Lemma \ref{Lem: enhanced Harnack} we conclude that the functions $w_s$ converge to a limit function $G_x^{(i)}$ satisfying $\Delta_g G_x^{(i)}=0$ in $\mathcal R\cap U_i$ and $G_x^{(i)}(o)=1$ up to a subsequence.

    Next we show the uniform estimates \eqref{Eq: lower bound} and \eqref{Eq: upper bound} for $G_x^{(i)}$. Note that we just need to prove the corresponding uniform estimates for $w_s$.
    Fix $r>0$ small. Then it follows from \cite[Lemma 3.4]{BBL20} that we have
   $$t\cdot\Capp(\{w_s\geq t\}, U_i)\equiv u_s(o)^{-1}\Capp(\overline{B_s(x)},U_i)\mbox{ for any }t\in (0,1].$$
   Using the Harnack inequality from Lemma \ref{Lem: ball harnack} we have $\lambda_0^{-1}\leq w_s\leq \lambda_0$ in some small ball $B_{\delta}(o)$ for some constant $\lambda_0>1$. From Lemma \ref{Lem: Sobolev} we can compute
   \[
   \begin{split}
       \lambda_0\Capp(\{w_s\geq \lambda_0\},U_i)&\geq\lambda_0^{-1}\int_{\mathcal R\cap U_i}|\nabla_g(\min\{w_s,\lambda_0\})|^2\,\mathrm d\mu\\
       &\geq \lambda_0^{-1}\left(\int_{\mathcal R\cap U_i}|(\min\{w_s,\lambda_0\})|^\frac{2m}{m-2}\,\mathrm d\mu\right)^{\frac{m-2}{m}}\\
       &\geq \lambda_0^{-3}\mu(B_{\delta}(o))^{\frac{m-2}{m}},
   \end{split}
   \]
 and this gives \eqref{Eq: lower bound} after taking limit. Denote $M_r=\sup_{\partial B_r(x)} w_s$ and $m_r=\inf_{\partial B_r(x)} w_s$. In particular, we have 
   $$M_r\geq c\Capp(\{w_s\geq M_r\},U_i)^{-1}.$$
   Using the local Ahlfors $m$-regularity, there is a constant $r_i>0$ such that we have
   $$\Capp(\{w_s\geq M_r\},U_i)\leq \Capp(B_r(x),B_{2r}(x))\leq \int_{B_r(x)}(\Lip\eta)^2\,\mathrm d\mu\leq Cr^{m-2}$$
   for any $r\in(0,r_i)$,
   where $\eta$ is the cut-off function such that $\eta\equiv 1$ in $B_{r}(x)$ and $\eta\equiv 0$ outside $B_{2r}(x)$. Combined with the Harnack inequality from Corollary \ref{Cor: annulus harnack}, we obtain 
   $$m_r\geq c_ir^{2-m}$$
   for any $r\in (0,r_i)$ and some positive constant $c_i$. Based on the Harnack inequality from Lemma \ref{Lem: enhanced Harnack} we have $$\sup_{\partial B_{r/2}(\mathcal S_i)} w_s\leq C_r$$
    for some constant $C_r>0$ independent of $x$. Let $\eta$ be a cut-off function such that $\eta\equiv 0$ in $B_{r/2}(\mathcal S_i)$ and $\eta\equiv 1$ outside $B_r(\mathcal S_i)$. From integration by parts and the maximum principle we obtain
    $$\int_{\mathcal R\cap (U_i-\overline{B_r(\mathcal S_i)})}|\nabla_g w_s|^2\,\mathrm d\mu\leq 4\int_{\mathcal R\cap (U_i-\overline{B_{r/2}(\mathcal S_i)})}|\nabla_g\eta|^2w_s^2\,\mathrm d\mu\leq \frac{16}{r^2C_r^2}\mu(U_i).$$
    By taking $C_i(r)=16r^{-2}C_r^{-2}\mu(U_i)$ we obtain \eqref{Eq: upper bound} after taking limit. Moreover, we obtain $\eta w_s\to \eta G_x^{(i)}\in W^{1,2}_0(U_i)$ for any cut-off function $\eta$ vanishing in a neighborhood of $\mathcal S_i$ up to a subsequence, and we complete the proof.
\end{proof}

The next lemma follows essentially from the nonlinear potential theory. Given any finite Radon measure $\nu$ on $X$, we introduce the {\it Wolff potential} $\mathcal W^{\nu}$ at any point $x$ by
\[
\mathcal W^{\nu}(x)
:=\int_0^1 \Bigl(\nu(B_r(x))\, r^{2-m}\Bigr)^{\frac{2}{m-2}}\,\mathrm dr.
\]

\begin{lemma}\label{lem:wolff-divergent-metric}
If we have
\[
\dim_{\mathcal H}(\mathcal S) < \frac{m-2}{2},
\]
then there exists a finite Radon measure $\nu$ supported on $\mathcal S_i$ such that
\[
\mathcal W^{\nu}(x)=+\infty
\mbox{ for every }x\in \mathcal S_i.
\]
\end{lemma}
\begin{proof}
Fix a constant $s\in (0,\frac{m-2}{2})$ such that
\[
\dim_{\mathcal H}(\mathcal S_i)< s \mbox{ and so }\mathcal H^{s}(\mathcal S_i)=0.
\]
For each integer $l\ge 1$ 
there exists a
countable covering $\{B_{r_{l,j}}(x_{l,j})\}_{j}$ of $\mathcal S_i$ with $x_{l,j}\in \mathcal S_i$
such that
\begin{equation}\label{Eq: small radius}
    r_{l,j}\le 2^{-l}
\mbox{ and }
\sum_j r_{l,j}^{\,s}\le 1
\end{equation}
Define the atomic measure
\[
\nu:=\sum_{l=1}^\infty \ \sum_{j} r_{l,j}^{\frac{m-2}{2}}\delta_{x_{l,j}}\mbox{ with } \operatorname{spt}\nu\subset \mathcal S_i,
\]
where $\delta_{x_{l,j}}$ denotes the Dirac measure supported at $x_{l,j}$.

We first check that $\nu$ is finite. For simplicity let us denote $\tau:=\frac{m-2}{2}-s>0$. From \eqref{Eq: small radius} we can compute
\[
r_{l,j}^{\frac{m-2}{2}}
= r_{l,j}^{\,s}\cdot r_{l,j}^{\tau}
\le  2^{-l\tau}\cdot r_{l,j}^{\,s},
\]
and then
\[
\sum_j r_{l,j}^{\frac{m-2}{2}}
\le 2^{-l\tau}\cdot\sum_j r_{l,j}^{\,s}
\le 2^{-l\tau}
\]
Therefore we obtain
\[
\nu(X)
=\sum_{l=1}^\infty\sum_j r_{l,j}^{\frac{m-2}{2}}
\le \sum_{l=1}^\infty 2^{-l\tau}
<\infty.
\]

Next we verify $\mathcal W^{\nu}(x)=+\infty$ for all $x\in \mathcal S_i$. Fix any $x\in K$. For each $l\ge 1$, choose an index $j'=j'(l)$ such that
$x\in B_{r_{l,j'}}(x_{l,j'})$. For any $r\in [r_{l,j'},2r_{l,j'}]$ we have
$$x_{l,j'}\in B_{r_{l,j'}}(x)\subset B_r(x)\mbox{ and so }\nu(B_r(x))\geq \sum_l r_{l,j'}^{\frac{m-2}{2}}\chi_{[r_{l,j'},2r_{l,j'}]}(r).$$
Therefore we can compute
\begin{align*}
\mathcal W^{\nu}(x)&=\int_0^1 \Bigl(\nu(B_r(x))\, r^{2-m}\Bigr)^{\frac{2}{m-2}}\,\mathrm dr\\
&\ge \sum_{l=1}^\infty \int_{r_{l,j'}}^{2r_{l,j'}}
\Bigl(r_{l,j'}^{\frac{m-2}{2}}\cdot r^{2-m}\Bigr)^{\frac{2}{m-2}}\,\mathrm dr\\
&\ge c\sum_{l=1}^\infty
\int_{r_{l,j'}}^{2r_{l,j'}} (r_{l,j'})^{-1} \,\mathrm dr
=+\infty.
\end{align*}
We complete the proof.
\end{proof}

Define
$$G_{\mathcal S_i}^{(i)}=\int_{\mathcal S_i} G_x^{(i)}\,\mathrm d\nu(x).$$

\begin{lemma}\label{Lem: Greens function at Si}
$G_{\mathcal S_i}^{(i)}:X-\mathcal S_i\to [0,+\infty)$ satisfies the following properties:
\begin{itemize}
    \item $\Delta_g G_{\mathcal S_i}^{(i)}=0$ in $\mathcal R\cap U_i$;
    \item $G_{\mathcal S_i}^{(i)}\in W^{1,2}(\mathcal R\cap U_i-\overline{B_r (\mathcal S_i)})$ for all $r>0$ and $\eta G_x^{(i)}\in W^{1,2}_0(U_i)$ for any cut-off function $\eta$ vanishing in a neighborhood of $\mathcal S_i$;
    \item we have
    $$\int_0^s \left(G_{\mathcal S_i}^{(i)}(\gamma(t))\right)^{\frac{2}{m-2}}\,\mathrm dt=+\infty$$
    for any unit-speed curve $\gamma:[0,s)\to\mathcal R\cap U_i$ with $\gamma(s^-)\in\mathcal S_i$.
\end{itemize}
\end{lemma}
\begin{proof}
    Using the Harnack inequality from Lemma \ref{Lem: enhanced Harnack}, the gradient estimates for harmonic functions, and the dominant convergence theorem as well as Lemma \ref{Lem: Greens function at x}, we obtain the first three properties. For the last property we can assume $s<r_i/2$ without loss  generality, where $r_i$ is the positive constant from Lemma \ref{Lem: Greens function at x}. Then we can compute
    \[
\begin{split}
   \int_0^{s}\left(G_{\mathcal S_i}^{(i)}(\gamma(t))\right)^{\frac{2}{m-2}}\,\mathrm dt
    &= \int_0^{s}\left(\int_{\mathcal S_i}G_x^{(i)}(\gamma(t))\,\mathrm d\nu\right)^{\frac{2}{m-2}}\,\mathrm dt\\
    &\geq \int_0^s\left(\nu(B_{s-t}(\gamma(s^-))\cdot c_i(2(s-t))^{2-m}\right)^{\frac{2}{m-2}}\,\mathrm dt\\
    &=\frac{1}{4}c_i^{\frac{2}{m-2}}(\mathcal W^\nu(\gamma(s^-))-C)=+\infty.
\end{split}
\]
We complete the proof.
\end{proof}

\begin{lemma}\label{Lem: G Si}
    If $(X,d,\mu)$ is non-parabolic, then we can find a function $G_{\mathcal S_i}:X-\mathcal S_i\to [0,+\infty)$ such that $\Delta_gG_{\mathcal S_i}=0$ in $\mathcal R$ and that we have 
    $$\int_0^s \left(G_{\mathcal S_i}(\gamma(t))\right)^{\frac{2}{m-2}}\,\mathrm dt=+\infty$$
    for any unit-speed curve $\gamma:[0,s)\to\mathcal R$ with $\gamma(s^-)\in\mathcal S_i$. Moreover, we have
    $$G_{\mathcal S_i}<C_{E,v}(v-\inf_X v)\mbox{ in }E,$$
    where $E$ is any open subset of $X$ with compact boundary $\partial E$ such that $\overline E\cap \mathcal S_i=\emptyset$, $v$ is any non-constant bounded superharmonic function on $X$, and $C_{E,v}$ is a positive constant depending on $E$ and $v$.
\end{lemma}
\begin{proof}
From Lemma \ref{Lem: Ui} we can construct an increasing bounded open exhaustion $\{V_l\}_{l=1}^\infty$ of $X$ such that $U_i\subset V_l$ and $\mathcal R\cap V_l$ is connected for all $l$. 
    Fix $r>0$ and we take a Lipschitz cut-off function such that $\eta\equiv 1$ in $B_r(\mathcal S_i)$ and $\eta\equiv 0$ outside $B_{2r}(\mathcal S_i)$. Define
    $$\mathbf F=\nabla_g((1-\eta)G_{\mathcal S_i}^{(i)}).$$
    It follows from Lemma \ref{Lem: Greens function at Si} that $\mathbf F\in L^2(V_l)$ for all $l\geq 1$. By Lemma \ref{Lem: possion equation} we can construct a function $u_l\in W^{1,2}_0(V_l)$ such that $\Delta_gu_l=\divv_g\mathbf F$ in $\mathcal R\cap V_l$. Let us consider the functions
    $$w_l=\eta G_{\mathcal S_i}^{(i)}+u_l.$$     Notice that we have $\Delta_g w_l=0$ in $\mathcal R\cap V_l$ from definition and also $w_{l+1}\geq w_l>0$ for all $l\in \mathbf N_+$ from the maximum principle. By comparison, we also have $w_1\geq G_{\mathcal S_i}^{(i)}$.
    
    We claim that $w_l(o)$ are uniformly bounded. Suppose otherwise that we have $w_l(o)\to +\infty$ as $l\to\infty$. Denote
    $$\bar w_l:=\frac{w_l}{w_l(o)}\mbox{ and } M_l:=\sup_{\partial V_1} \bar w_l.$$
    Using the Harnack inequality from Lemma \ref{Lem: enhanced Harnack} we can assume $\bar w_l\to \bar w_\infty$ and $M_l\to M_\infty$ as $l\to \infty$ up to a subsequence, where we have $\bar w_\infty\in W^{1,2}_{loc}(X-\mathcal S_i)$, $\Delta_g\bar w_\infty=0$ in $\mathcal R$ and $M_\infty\in (0,+\infty)$. From the maximum principle, we actually have
    $$\frac{w_1}{w_l(o)}\leq \bar w_l\leq \frac{w_1}{w_l(o)}+M_l\mbox{ in }V_1\mbox{ and }\bar w_l\leq M_l\mbox{ in }V_l-\bar V_1.$$
    Passing to the limit, we conclude that $\bar w_\infty$ attains its maximum value $M_\infty$ in $\partial V_1$ and then the strong maximum principle yields $\bar w_\infty\equiv M_\infty$. Let $v$ be a non-constant bounded superharmonic function on $X$. From the strong maximum principle we have $v-\inf_X v\geq c$ on $\partial V_1$ for some positive constant $c$. By comparison we see $\bar w_l\leq c^{-1}M_l(v-\inf_Xv)$ in $V_l-\bar V_1$ and so we obtain
    $$\bar w_\infty\leq c^{-1}M_\infty (v-\inf_Xv).$$
    This yields $\inf_X\bar w_\infty=0$, which contradicts to the facts that $\bar w_\infty$ is a constant function and $\sup_{\partial V_1}\bar w_\infty=M_\infty>0$.

    Since $w_l(o)$ are uniformly bounded, using the Harnack inequality from Lemma \ref{Lem: enhanced Harnack} we can deduce that the functions $w_l$ converge to a limit function $G_{\mathcal S_i}$ satisfying $\Delta_gG_{\mathcal S_i}=0$ in $\mathcal R$. Clearly we have $G_{\mathcal S_i}\geq G_{\mathcal S_i}^{(i)}$ and so from Lemma \ref{Lem: Greens function at Si} we have 
    $$\int_0^s \left(G_{\mathcal S_i}(\gamma(t))\right)^{\frac{2}{m-2}}\,\mathrm dt=+\infty$$
    for any unit-speed curve $\gamma:[0,s)\to\mathcal R$ with $\gamma(s^-)\in\mathcal S_i$. The last statement comes from a similar comparison argument as above.
\end{proof}

\begin{proof}[Proof of Proposition \ref{Prop: weak form}]
    We can define
    $$G_{\mathcal S}=\sum_{i=1}^\infty 2^{-i}\frac{G_{\mathcal S_i}}{G_{\mathcal S_i}(o)}.$$
    Clearly, we have $G_{\mathcal S}(o)=1$ and 
    $$G_{\mathcal S}<C_{E,v}(v-\inf_X v)\mbox{ in }E.$$
    From the Harnack inequaity, the gradient estimates, and the dominant divergence theorem, we conclude that $G_{\mathcal S}$ is a smooth positive harmonic function on $\mathcal R$. Let us verify that the conformal metric
\[
\tilde g := (1+\delta G_{\mathcal S})^{\frac{4}{m-2}} g
\]
 is a complete Riemannian metric on $\mathcal R$. Otherwise, by the Hopf-Rinow theorem we can find a unit-speed $\tilde g$-geodesic $\gamma:[0,\tilde s)\to \mathcal R$ with $\tilde s<+\infty$ which cannot be extended any more. Notice that we have $\tilde g\geq g$, and so the completeness of $(X,d)$ combined with the non-extended property of $\gamma$ yields $\gamma(\tilde s^-)\in \mathcal S$. In particular, we can find $i\in\mathbf N_+$ such that $\gamma(\tilde s^-)\in \mathcal S_i$. In the following, we work with the $g$-unit-speed reparametrization $\gamma:[0,s)\to \mathcal R$. From Lemma \ref{Lem: G Si} we can compute
 $$\tilde s=\int_0^s(G_{\mathcal S}(\gamma(t)))^{\frac{2}{m-2}}\,\mathrm dt\geq \left(\frac{\delta}{2^i G_{\mathcal S_i}(o)}\right)^{\frac{2}{m-2}}\int_0^s\left(G_{\mathcal S_i}(\gamma(t))\right)^{\frac{2}{m-2}}\,\mathrm dt=+\infty,$$
 which leads to a contradiction.
\end{proof}

\subsection{Modification and application}

First let us present the following modification of Proposition \ref{Prop: weak form}.
\begin{proposition}\label{Prop: strong form}
    If $(X,d,\mu)$ is non-parabolic and $\mathcal S$ satisfies
    $$\mathcal H^{\frac{m-2}{2}}(\mathcal S\cap U)<+\infty\mbox{ for any bounded }U\subset X.$$
    Moreover, we assume that $(X,d)$ can be bi-Lipschitzly embedded into a smooth Riemannian manifold. Then there is a smooth harmonic function $G_{\mathcal S}$ on $\mathcal R$ such that the conformal metric
    $$\tilde g:=(1+\delta G_{\mathcal S})^{\frac{4}{m-2}}g$$
    is complete on $\mathcal R$ for any $\delta>0$. Moreover, we have
    $$G_{\mathcal S}<C_{E,v}(v-\inf_X v)\mbox{ in }E,$$
    where $E$ is any open subset of $X$ with compact boundary $\partial E$ such that $\overline E\cap \mathcal S=\emptyset$, $v$ is any non-constant bounded superharmonic function on $X$, and $C_{E,v}$ is a positive constant depending on $E$ and $v$.
\end{proposition}
This follows from the same proof as before, with Lemma \ref{lem:wolff-divergent-metric} replaced by the following one.

\begin{lemma}\label{Lem: Wolff divergent}
    Lemma \ref{lem:wolff-divergent-metric} is still true if $$\mathcal H^{\frac{m-2}{2}}(\mathcal S\cap U)<+\infty\mbox{ for any bounded }U\subset X.$$
     and $(X,d)$ can be bi-Lipschitzly embedded into a smooth Riemannian manifold $(M^N,h)$.
\end{lemma}
\begin{proof}
Since the embedding $(X,d)\hookrightarrow (M^N,h)$ is bi-Lipschitz, the distances $d$ and $d_h$ on $X$ are comparable, where $d_h$ denotes the Riemannian distance induced from $(M^N,h)$. As a consequence, by taking the parameters $(\alpha,p)$ such that
\[
N-\alpha p = \frac{m-2}{2}\mbox{ and }
\frac{1}{p-1}=\frac{2}{m-2},
\] the
Wolff potential $\mathcal W^{\nu}_{\alpha,p}$ from \cite{AH96} with a finite Radon measure $\nu$ supported on $X$ is comparable to the Wolff potential
$\mathcal W^{\nu}$. Therefore, we just need to construct a finite Radon measure in $(M^N,h)$ supported on $\mathcal S_i$ such that $\mathcal W^{\nu}_{\alpha,p}(x)=+\infty$ for all $x\in \mathcal S_i$.

Use the same notation $\mathcal S_i$ as before. By assumption we have
$$\mathcal H^{\frac{m-2}{2}}(\mathcal S_i)<+\infty.$$
Let $C_{\alpha,p}$ denote the Bessel capacity on $M^N$. According to \cite[Theorem 5.19]{AH96}  we have
\[
C_{\alpha,p}(\mathcal S_i)=0.
\]
By the dual characterization of $C_{\alpha,p}$ together with Wolff's inequality
\cite[Theorem~4.5.2]{AH96}, the identity $C_{\alpha,p}(\mathcal S_i)=0$ implies that
there exists a finite Radon measure $\nu$ supported on $\mathcal S_i$ such that $\mathcal W^{\nu}_{\alpha,p}(x)=+\infty$ for all $x\in \mathcal S_i$. This completes the proof.
\end{proof}

Proposition \ref{Prop: desingularization} follows immediately.
\begin{proof}[Proof of Proposition \ref{Prop: desingularization}]
    It follows from Proposition \ref{Prop: weak form} and Proposition \ref{Prop: strong form}.
\end{proof}

For our later use, we only need the following special version of singularity blow-up.

\begin{corollary}\label{Cor: singularity blowup}
   Let $(M^{N},g,E)$ be a $\mathbf T^k$-warped ALF manifold with $N-k\geq 4$ and $\tilde\Sigma^{N-1}$ be a connected area-minimizing hypersurface with singular set $\mathcal S$ and a $\mathbf T^k$-warped ALF end $\tilde E$. If we have
   $$\mathcal H^{\frac{N-3}{2}}(\mathcal S\cap U)<+\infty\mbox{ for any bounded }U\subset \tilde \Sigma.$$
   there is a smooth harmonic function $G_{\mathcal S}$ on $\mathcal R$ such that the conformal metric
    $$\tilde g:=(1+\delta G_{\mathcal S})^{\frac{4}{m-2}}g$$
    is complete on $\mathcal R$ for any $\delta>0$, where $\mathcal R=\tilde \Sigma-\mathcal S$. Moreover, we have $G_{\mathcal S}\to 0$ as $x\to \infty$ in the ALF end $\tilde E$.
\end{corollary}
\begin{proof}
    Define
    $$X=\tilde\Sigma,\,d=d_g|_{\tilde \Sigma}\mbox{ and }\mu=\mathcal H^{N-1}_g|_{\tilde \Sigma}.$$
    Then $(X,d,\mu)$ is a metric measure space. Since $\tilde \Sigma$ is area-minimizing, it follows from the monotonicity formula and local area comparison that $(X,d,\mu)$ satisfies the local Ahlfors $(N-1)$-regularity. Then
    it follows from \cite{BG1972} and the assumption that $(X,d,\mu)$ satisfies (AM) and (LPI). It remains to show that $(X,d,\mu)$ is non-parabolic. By definition, $\tilde E$ is diffeomorphic to $(\mathbf R^{N-k-1}-B_{r_0})\times \mathbf T^k$, where the coordinate will be denoted by $(x',\theta)$. From a direct computation, we have $\Delta |x'|^{-1/2}<0$ around the infinity of $E$ and so $\min\{|x'|^{-1/2},\delta\}$ gives the desired superharmonic function on $(X,d,\mu)$ if we choose $\delta>0$ to be small enough. The result follows from Proposition \ref{Prop: desingularization}.
\end{proof}

\section{Torical symmetrization with negative mass}

In this section, we always use the notation
$$u=\mathcal O_{l}(|x|^{-\delta})\mbox{ as }x\to \infty$$
to mean that a smooth function $u$ on a $\mathbf T^k$-warped ALF manifold satisfies
$$\sum_{j=0}^l|x|^j|\partial^ju|=O(|x|^{-\delta})\mbox{ as }x\to \infty.$$
We also write $u=\mathcal O_\infty(|x|^{-\delta})$ if $u=\mathcal O_l(|x|^{-\delta})$ for all $l\in\mathbf N$.
For convenience, we denote
$$N=n+k\mbox{ and }\tau_*=\min\{\tau,1\}.$$

\subsection{Asymptotic modification}

\begin{proposition}\label{prop: density}
    Given any $\mathbf T^k$-warped ALF manifold $(M^{n+k},g,E)$ with negative mass and nonnegative scalar curvature, we can find a new $\mathbf T^k$-warped ALF manifold $(M,\tilde g,\tilde E)$ with negative mass and nonnegative scalar curvature such that the metric $\tilde g$ in $\tilde E$ has the form
    $$\tilde g=v^{\frac{4}{N-2}}\bar g,$$
    where $v$ is $\mathbf T^k$-invariant and has the expansion
    $$ v=1+A|x|^{2-n}+\mathcal O_\infty(|x|^{1-n})\mbox{ as }x\to \infty\mbox{ with }A<0$$
    for any constant $\delta \in (0,\tau_*)$, and $\bar g$ is the reference metric in Definition \ref{def: ALF}. Moreover, we can guarantee $R_{\tilde g}>0$ in any fix annulus $\mathcal A_{r,4r}=\{(x,\theta)\in \tilde E:r\leq|x| \leq 4r\}$ by arbitrarily small perturbation.
\end{proposition}

First we recall the following existence lemma from \cite{CLSZ25}. 

\begin{lemma}\label{Lem: existence f eigenfunction}
    Given any smooth open $\mathbf T^k$-invariant region $U\subset M$ with $\overline{U\Delta E}$ compact, there is a universal constant $c=c(U)$ such that if $f$ is a smooth function on $M$ with compact support in $U$ satisfying
    $$\left(\int_U|f_-|^\frac{N}{2}|x|^k\,\mathrm d \mu_g\right)^{\frac{2}{N}}\leq c,$$
    then the equation
    $$\Delta_gu=fu$$
    has a uniformly positive solution $u$ on $M$. Moreover, $u$ has the expansion
    \begin{equation}\label{Eq: expansion}
        u=1+A|x|^{2-n}+\mathcal O_2(|x|^{2-n-\delta})\mbox{ as }x\to \infty
    \end{equation}
    for any constant $\delta \in (0,\tau_*)$,
    where
    \begin{equation}\label{Eq: coefficient A}
        A=-\frac{1}{(n-2)\omega_{n-1}}\int_Ufu\,\mathrm d\mu_g,
    \end{equation}
    and we have
    $$\int_{\partial U}\nabla_g u\cdot \nu_{\partial U}\,\mathrm d\sigma_g=0\mbox{ and }\int_{\partial U}u\nabla_gu\cdot \nu_{\partial U}\,\mathrm d\sigma_g\leq 0,$$
    where $\nu_{\partial U}$ denotes the unit outward normal of $\partial U$. If $f$ is a $\mathbf T^k$-invariant function, then $u$ can be required to be $\mathbf T^k$-invariant as well.
\end{lemma}
\begin{proof}
    This is \cite[Proposition 4.10]{CLSZ25}. We sketch its proof for completeness. It follows from \cite[Lemma 4.9]{CLSZ25} that we have the weighted Sobolev inequality
    \begin{equation}\label{Eq: Sobolev}
        \left(\int_U|\zeta|^{\frac{2N}{N-2}}|x|^{\frac{-2k}{N-2}}\,\mathrm d\mu_g\right)^{\frac{N-2}{2N}}\leq C_U\left(\int_U|\nabla_g\zeta|^2\,\mathrm d\mu_g\right)^{\frac{1}{2}}
    \end{equation}
    for some universal positive constant $C_U$ independent of $\zeta$, where the test function $\zeta$ is allowed to take non-zero values on $\partial U$. Set
    $$c=(2C_U)^{-1}.$$
    Then the operator $-\Delta_g+f$ is positive in any compact region. Fix an exhaustion $\{U_i\}_{i\geq 1}$ of $M$, where $U_i$ are $\mathbf T^k$-invariant smooth open subsets containing $U$ and $U_i\Delta U$ are bounded. In particular, we can solve the following PDE
    $$\Delta_gv_i-fv_i=f\mbox{ in }U_i$$
    with the Neumann-Dirichlet boundary condition
    $$\nabla_gv_i\cdot \nu_{\partial U_i}=0\mbox{ on }\partial U_i\mbox{ and }v_i(x,\theta)\to 0\mbox{ as }x\to \infty. $$
    From integration by parts and the weighted Sobolev ineqality \eqref{Eq: Sobolev}, it is easy to deduce the uniform estimate
    \begin{equation}\label{Eq: uniform estimate}
        \left(\int_{U}|v_i|^{\frac{2N}{N-2}}|x|^{\frac{-2k}{N-2}}\,\mathrm d\mu_g\right)^{\frac{N-2}{2N}}\leq 2C_U^2\left(\int_U|f|^{\frac{2N}{N+2}}|x|^{\frac{2k}{N+2}}\,\mathrm d\mu_g\right)^{\frac{N+2}{2N}}.
    \end{equation}
    By standard theory for elliptic PDEs we have
    $$\|v_i\|_{C^l(K)}\leq C(l,K)$$
    for any given integer $l\in \mathbf N_+$ and compact subset $K$. Up to a subsequence, the functions $v_i$ converge to a smooth function $v$ on $M$. 
    
    Define $u=v+1$. Then $u$ solves the equation $\Delta_g u=fu$. The positivity of $-\Delta_g+f$ yields $v_i+1\geq 0$ and so we have $u> 0$ from the convergence and the strong maximum principle. Note that $v_i$ are harmonic functions outside $U$. The maximum principle yields
    $$\inf_{U_i} v_i\geq \inf_U v_i\mbox{ and so }\inf_Mu\geq \inf_{U}u>0. $$
    From integration by parts and the divergence theorem we see
    $$\int_{\partial U}\nabla_g u\cdot \nu_{\partial U}\,\mathrm d\sigma_g=\lim_{i\to\infty}\int_{\partial U}\nabla_gv_i\cdot \nu_{\partial U}\,\mathrm d\sigma_g=\lim_{i\to\infty}\int_{\partial U_i}\nabla_g v_i\cdot \nu_{\partial U_i}\,\mathrm d\sigma_g=0$$
    and
    \[
    \begin{split}
\int_{\partial U}u\nabla_g u\cdot \nu_{\partial U}\,\mathrm d\sigma_g&=\lim_{i\to\infty}\int_{\partial U}(v_i+1)\nabla_g v_i\cdot \nu_{\partial  U}\,\mathrm d\sigma_g \\
&\leq \lim_{i\to \infty}\int_{\partial U_i}(v_i+1)\nabla_g v_i\cdot \nu_{\partial U_i}\,\mathrm d\sigma_g=0,
    \end{split}
    \]
    where we use the boundary condition $\nabla_g v_i\cdot \nu_{\partial U_i}=0$ on $\partial U_i$. From a comparison with the power functions, it is easy to obtain $$v=O(|x|^{2-n+\tau_*-\delta})$$ for any constant $\delta\in (0,\tau_*)$, and then the Schauder estimate yields $v=\mathcal O_2(|x|^{2-n+\tau_*-\delta})$ as $x\to \infty$. In particular, we have
    $$\Delta_{\bar g}v=O(|x|^{-n-\delta})\mbox{ as }x\to \infty,$$
    where $\bar g$ denotes the reference metric from Definition \ref{def: ALF}. Define
    $$w=\int_{\mathbf T^k}v(x,\theta)\,\mathrm d\theta.$$
    Then $w$ is a decaying function defined around the infinity of $\mathbf R^n$ and we have
    $$\Delta_{\mathbf R^n}w=O(|x|^{-n-\delta})\mbox{ as }x\to\infty.$$
    In particular, we have the expansion
    $$w=A|x|^{2-n}+O(|x|^{2-n-\delta}).$$
    Now we consider the function $\bar v:=v-A|x|^{2-n}$. Through a direct computation, we have
$$|v(x,\theta)-w(x)|\leq \sup_{\{x\}\times \mathbf T^k}|\partial v|\cdot \diam(\mathbf T^k,\mathrm d\theta\otimes \mathrm d\theta)=O(|x|^{1-n})$$
    and so
    $$\Delta_g\bar v=-A\Delta_g|x|^{2-n}=\mathcal O_1(|x|^{-n-\delta})\mbox{ and }\bar v=O(|x|^{2-n-\delta}).$$
    From the Schauder estimate we obtain the desired expansion \eqref{Eq: expansion}. The constant $A$ can be determined by integrating the equation $\Delta_gu=fu$.
    
    If $f$ is a $\mathbf T^k$-invariant function on $M$, we can solve $v_i$ in the class of $\mathbf T^k$-invariant functions since the metric $g$ is also $\mathbf T^k$-invariant. Then $u$ is $\mathbf T^k$-invariant as desired.
\end{proof}

\begin{proof}[Proof of Proposition \ref{prop: density}]
We follow the proof of \cite[Proposition 4.11]{CLSZ25}. In the ALF end $E$, we can write $g$ as
$$g=\left(1+A_0|x|^{2-n}\right)^{\frac{4}{N-2}}\bar g+Q\mbox{ with }A_0:=\frac{N-2}{2N(N-1)}m(M,g,E).$$
Through a direct computation, we have
\begin{equation}\label{Eq: zero mass difference}  \lim_{\rho\to+\infty}\int_{\mathbf S^{n-1}(\rho)\times \mathbf T^k}(\partial_iQ_{ij}-\partial_jQ_{\alpha\alpha})\nu^j\,\mathrm d\sigma_{\bar g}=0.
\end{equation}
For any constant $s>0$ we can define
$$\hat g=\left(1+A_0|x|^{2-n}\right)^{\frac{4}{N-2}}\bar g+\left(1-\zeta\left(\frac{|x|}{s}\right)\right)Q,$$
where $\zeta:\mathbf R\to [0,1]$ is a cut-off function such that $\zeta\equiv 0$ in $(-\infty,2]$ and $\zeta \equiv 1$ in $[3,+\infty)$.

Fix another cut-off function $\eta:\mathbf R\to [0,1]$ such that $\eta\equiv 0$ outside $[1,4]$ and $\eta\equiv 1$ in $[2,3]$. 
    Denote
    $$f=\frac{N-2}{4(N-1)}\eta\left(\frac{|x|}{s}\right)R_{\hat g}.$$
    Recall that we denote $\mathcal A_{r_1,r_2}=\{(x,\theta)\in E:r_1\leq |x|\leq r_2\}$ for any constants $r_1<r_2$. Note that $R_{\hat g}$ is nonnegative outside $\mathcal A_{2s,3s}$. Therefore, we have
    $$\int_E|f_-|^{\frac{N}{2}}|x|^k\,\mathrm d\mu_{\hat g}\leq O(s^k)\cdot\int_{\mathcal A_{2s,3s}}|R_{\hat g}|^{\frac{N}{2}}\,\mathrm d\mu_{\hat g}=O(s^{-\frac{N\tau}{2}})\to 0\mbox{ as }s\to+\infty.$$
    By taking $s$ large enough, it follows from Lemma \ref{Lem: existence f eigenfunction} that we can find a uniformly positive $\mathbf T^k$-invariant function $u$ on $M$ solving the equation $\Delta_{\hat g}u=fu$ with the expansion
    $$u=1+A_1|x|^{2-n}+\mathcal O_2(|x|^{2-n-\delta})$$ Define
    $$\tilde g=u^{\frac{4}{N-2}}\hat g.$$
    If we take $\tilde E=\{(x,\theta)\in E:|x|\geq 4s\}$, then we have
    $$\tilde g=v^{\frac{4}{N-2}}\bar g\mbox{ in }\tilde E, \mbox{ where }v=u\left(1+A_0|x|^{2-n}\right).$$
    Note that $v$ is harmonic in $(\tilde E,\bar g)$, so the function $v$ has the expansion
    $$ v=1+A|x|^{2-n}+\mathcal O_\infty(|x|^{1-n})\mbox{ as }x\to \infty \mbox{ with }A=A_0+A_1.$$

    We claim $A<0$ when $s$ is large enough. Recall from \eqref{Eq: coefficient A} that we have
    $$A_1=-\frac{1}{(n-2)\omega_{n-1}}\int_Ufu\,\mathrm d\mu_{\hat g}.$$
    Let us write
    $$\int_Ufu\,\mathrm d\mu_{\hat g}=\int_{\mathcal A_{s,4s}}fv\,\mathrm d\mu_{\hat g}+\int_{\mathcal A_{s,4s}}f\,\mathrm d\mu_{\hat g}.$$
    From the H\"older inequality and the uniform estimate \eqref{Eq: uniform estimate} we have
    \[
    \begin{split}
        \left|\int_{\mathcal A_{s,4s}}fv\,\mathrm d\mu_{\hat g}\right|&\leq C \left(\int_{\mathcal A_{s,4s}}|R_{\hat g}|^{\frac{2N}{N+2}}|x|^{\frac{2k}{N+2}}\,\mathrm d\mu_{\hat g}\right)^{\frac{N+2}{2N}}\left(\int_{\mathcal A_{s,4s}}|v|^{\frac{2N}{N-2}}|x|^{-\frac{2k}{N-2}}\,\mathrm d\mu_{\hat g}\right)^{\frac{N-2}{2N}}\\
        &\leq C \left(\int_{\mathcal A_{s,4s}}|R_{\hat g}|^{\frac{2N}{N+2}}|x|^{\frac{2k}{N+2}}\,\mathrm d\mu_{\hat g}\right)^{\frac{N+2}{N}}=O(s^{n-2-2\tau})\to 0\mbox{ as }s\to +\infty.
    \end{split}
    \]  
    Integrating $f$ over the annulus $\mathcal A_{s,4s}$ and using the nonnegativity of $R_{\hat g}$ outside $\mathcal A_{2s,3s}$, we see
    $$\frac{N-2}{4(N-1)}\int_{\mathcal A_{2s,3s}}R_{\hat g}\,\mathrm d\mu_{\hat g}\leq\int_{\mathcal A_{s,4s}}f\,\mathrm d\mu_{\hat g}\leq \frac{N-2}{4(N-1)}\int_{\mathcal A_{s,4s}}R_{\hat g}\,\mathrm d\mu_{\hat g}.$$
    In the coordinate chart $(x,\theta)$, the scalar curvature has the expansion
$$R_{\hat g}=|\hat g|^{-\frac{1}{2}}\partial_\alpha(\partial_\beta \hat g_{\alpha\beta}-\partial_\alpha \hat g_{\beta\beta})+O(|x|^{-2-2\tau})\mbox{ as }x\to \infty.$$
    Integrating $R_{\hat g}$ and using the limit \eqref{Eq: zero mass difference}, we can deduce
    $$\lim_{s\to+\infty}\int_{\mathcal A_{2s,3s}}R_{\hat g}\,\mathrm d\mu_{\hat g}=\lim_{s\to+\infty}\int_{\mathcal A_{s,4s}}R_{\hat g}\,\mathrm d\mu_{\hat g}=0,$$
    which implies
    $$\int_{\mathcal A_{s,4s}}f\,\mathrm d\mu_{\hat g}\to 0\mbox{ as }s\to+\infty.$$
    We have shown $A_1\to 0$ as $s\to +\infty$. By definition and our assumption we have $A_0<0$. So we can guarantee $A<0$ when $s$ is large enough.

    Let us show that $(M,\tilde g,\tilde E)$ is a $\mathbf T^k$-warped ALF manifold with negative mass and nonnegative scalar curvature. First we point out that $\tilde g$ is a complete metric since the conformal factor $u$ is uniformly positive. From assumption we can write $M=M_0\times \mathbf T^k$ and 
    $$g=g_0+\sum_{i=1}^ku_i^2\mathrm d\theta_i\otimes \mathrm d\theta_i.$$
    Through a direct computation we see
    $$\tilde g=\tilde g_0+\sum_{i=1}^k\tilde u_i^2\mathrm d\theta_i\otimes \mathrm d\theta_i,$$
    where
    $$\tilde g_0=u^{\frac{4}{N-2}}\left(\left(1-\zeta\left(\frac{|x|}{s}\right)\right)g_0+\zeta\left(\frac{|x|}{s}\right)(1+A_0|x|^{2-n})^{\frac{4}{N-2}}(\mathrm dx\otimes \mathrm dx)\right)$$
    and
    $$\tilde u_i=u^{\frac{2}{N-2}}\left(\left(1-\zeta\left(\frac{|x|}{s}\right)\right)u_i+\zeta\left(\frac{|x|}{s}\right)(1+A_0|x|^{2-n})^{\frac{2}{N-2}}\right).$$
    Recall that we have $\tilde g=v^{\frac{4}{N-2}}\bar g$ in $\tilde E$. From the expansion of $v$ we know that $(\tilde E,\tilde g)$ satisfies the ALF decay condition \eqref{Eq: decay condition}. By direct computation, we have $R_{\tilde g}=O(|x|^{-n})$ as $x\to \infty$ and so $R_{\tilde g}\in L^1(\tilde E,\tilde g)$. The mass $m(M,\tilde g,\tilde E)$ is a positive multiple of $A$, which is negative due to $A<0$. The nonnegativity of $R_{\tilde g}$ follows from
    $$R_{\tilde g}=u^{-\frac{4}{N-2}}\left(R_{\hat g}-\eta\left(\frac{|x|}{s}\right)R_{\hat g}\right)\geq 0.$$
    
     Next we show how to make arbitrarily small perturbation to guarantee $R_{\tilde g}>0$ in any fixed annulus $\mathcal A_{r,4r}$. From Lemma \ref{Lem: existence f eigenfunction} by taking a very small nonnegative $\mathbf T^k$-invariant cut-off function $\eta$ such that $\eta> 0$ in $\mathcal A_{r,4r}$ and $\eta\equiv 0$ outside a small neighborhood of $\mathcal A_{r,4r}$,  we can find a smooth positive function $w_0$ solving the equation
     \begin{equation}\label{Eq: equation w_0}
         \Delta_{\tilde g}w_0=-\frac{N-2}{4(N-1)}\eta w_0
     \end{equation}
     with the expansion
     $$w_0=1+A_3|x|^{2-n}+\mathcal O_2(|x|^{2-n-\delta})\mbox{ as }x\to \infty.$$
     Define
     $$w=\frac{1+\varepsilon w_0}{1+\varepsilon}\mbox{ with }\varepsilon\in (0,1)$$
     and take $w^{\frac{4}{N-2}}\tilde g$ to be the new metric. Then the new scalar curvature is given by
     $$R_{w^{\frac{4}{N-2}}\tilde g}=w^{-\frac{N+2}{N-2}}\left(\frac{\varepsilon\eta}{1+\varepsilon}w_0 +R_{\tilde g}w\right),$$
     which is positive in the annulus $\mathcal A_{r,4r}$. 
     The mass $m(M,w^{\frac{4}{N-2}}\tilde g,E)$ differs from $m(M,\tilde g,E)$ by a constant multiple of $\varepsilon$. Therefore, we can ensure that the new mass is still negative. The new conformal factor $wv$ still has the desired expansion since it is also harmonic with respect to $\bar g$ around the infinity of $\tilde E$.
\end{proof}

\subsection{Construction of strongly stable area-minimizing hypersurface}\label{Sec: existence minimizing}

Given any $\mathbf T^k$-warped Riemannian manifold $(M,g)$ with a $\mathbf T^k$-warped ALF end $E$, we denote the test function space
$$\mathcal T_{-q}(M,g,E)=\left\{1+\psi:\psi\in\bigcup_l W^{1,2}_{-q}(U_l,g)\right\},$$
where $q>0$ and $\{U_l\}_{l\geq 1}$ is an increasing open exhaustion of $M$ where $\overline{U_l\Delta E}$ is compact in $M$ for all $l\geq 1$.

\begin{proposition}\label{Prop: minimizing hypersurface}
    Let $(M^{n+k},g,E)$ be a $\mathbf T^k$-warped ALF manifold with
    \begin{equation}\label{Eq: Schwarzschild type metric}
        g=v^{\frac{4}{N-2}}\bar g \mbox{ in }E,
    \end{equation}
    where $v$ has the expansion
    $$ v=1+A|x|^{2-n}+\mathcal O_\infty(|x|^{1-n})\mbox{ with }A<0.$$
    Then there is a $\mathbf T^k$-invariant area-minimizing hypersurface $\tilde\Sigma$ 
    satisfying the following properties:
    \begin{itemize}
        \item[(i)] $\tilde\Sigma$ possibly has non-empty singular set $\mathcal S$ satisfying $$\mathcal H^{N-8}(\mathcal S\cap K)<+\infty$$ for any compact subset $K$;
        \item[(ii)] $\tilde\Sigma$ is a smooth minimal graph over $\mathbf R^{n-1}\times \mathbf T^k$ around infinity, where the graph function $w$ has the expansion
        $$w(x',\theta)=c_0+c_1|x'|^{3-n}+\mathcal O_\infty(|x'|^{2-n})\mbox{ as }x'\to\infty.$$
        Let $E_{\tilde\Sigma}$ be the minimal graph part of $\Sigma$ and $g_{\tilde\Sigma}$ be the induced metric of $\tilde \Sigma-\mathcal S$. In particular, $(\tilde\Sigma-\mathcal S,g_{\tilde \Sigma})$ is a $\mathbf T^k$-warped Riemannian manifold with a $\mathbf T^k$-warped ALF end $E_{\tilde \Sigma}$;
        \item[(iii)] $\tilde\Sigma$ is strongly stable in the sense that for any test function $$\phi\in \mathcal T_{-q}(\tilde\Sigma-\mathcal S,g_{\tilde \Sigma},E_{\tilde \Sigma})\mbox{ with }q>\frac{n-3}{2},$$ 
        we have
        $$\int_{\tilde \Sigma}\left(\Ric_g(\tilde \nu)+|\tilde A_{\tilde \Sigma}|^2\right)\phi^2\,\mathrm d\mathcal H^{N-1}\leq\int_{\tilde \Sigma}|\nabla_{\tilde \Sigma}\phi |^2\,\mathrm d\mathcal H^{N-1}.$$
    \end{itemize}
\end{proposition}

The construction is well-known since  \cite{Schoen1989}. Here we will set free-boundary minimizing problems as in \cite{HSY26} instead of the original leafwise Plateau problem.

In the following, let us write $x=(x',x_{n})$. For any positive constants $s$ and $t$ we define
$$\mathcal C_s=\{(x,\theta)\in E:|x'|< s\}\cup (M-E)$$
and
$$\mathcal P_t=\{(x,\theta)\in E: x_n=t\}\mbox{ and }\mathcal Q_t=\{(x,\theta)\in E: x_n\leq t\}.$$

\begin{lemma}\label{Lem: barrier}
    There is a positive constant $\Lambda$ such that $\mathcal P_t$ is mean-convex with respect to $e_n$ and $\mathcal P_{-t}$ is mean-convex with respect to $-e_n$ whenever $t\geq \Lambda$. Moreover, $\mathcal P_t$ is always perpendicular to $\mathcal C_s$.
\end{lemma}
\begin{proof}
    Let us compute the mean curvature of $\mathcal P_t$ with respect to $e_n$. It is clear that $\mathcal P_t$ has  vanishing $\bar g$-mean curvature. Then the $g$-mean curvature is given by
    $$H_g=\frac{2(N-1)}{N-2}v^{-\frac{N}{N-2}}\nabla_{\bar g}v\cdot e_n.$$
    Since we have
    $$\nabla_{\bar g}v\cdot e_n=(2-n)A|x|^{-n}x_n+\mathcal O_\infty(|x|^{-n}),$$
    we can take the constant $\Lambda$ large enough such that the $g$-mean curvature of $\mathcal P_t$ with $|t|\geq \Lambda$ has the same sign as $x_n$. Since $g$ is conformal to $\bar g$, $\mathcal P_t$ is always perpendicular to $\mathcal C_s$.
\end{proof}

\begin{remark}\label{Rem: stability of barrier}
    We remark that the constant $\Lambda$ above can be taken large enough such that the following property holds: for any smooth function $w_0$ with expansion
    \begin{equation}\label{Eq: expansion w0}
        w_0=1+A_3|x|^{2-n}+\mathcal O_\infty(|x|^{1-n})\mbox{ as }x\to \infty,
    \end{equation}
    there is a constant $\varepsilon_0$ such that Lemma \ref{Lem: barrier} holds with $\Lambda$ for any conformal metric in the form of
    $$g_\varepsilon=\left(\frac{1+\varepsilon w_0}{1+\varepsilon}\right)^{\frac{4}{N-2}}g\mbox{ with }|\varepsilon|<\varepsilon_0.$$
    Note that the function $w_0$ constructed by \eqref{Eq: equation w_0} has the expansion \eqref{Eq: expansion w0} since both $v$ and $w_0v$ there are harmonic functions with such expansions. In particular, the perturbation in Proposition \ref{prop: density} can be made without affecting the barrier constant $\Lambda$ in Lemma \ref{Lem: barrier}.
\end{remark}

Denote the reference set
$$\bar F=\{(x,\theta)\in E:x_{n}\leq 0\}.$$
Define
$$\mathcal F=\{\mathbf T^k\mbox{-invariant Caccioppoli sets }F\mbox{ such that }(F\Delta \bar F)\cap E\mbox{ is bounded}\}.$$
Fix a sequence of constants $s_i\to+\infty$ as $i\to\infty$ and write $\mathcal C_i=\mathcal C_{s_i}$ for short. Set
$$\mu_i=\inf_{F\in\mathcal F}P(F,\mathcal C_i),$$
where $P(F,\mathcal C_i)$ denotes the perimeter of $F$ in $\mathcal C_i$. Roughly speaking, the perimeter $P(F,\mathcal C_i)$ measures the area of $\partial F$ in $\mathcal C_i$.

\begin{lemma}\label{Lem: minimizing exhaustion}
    There is $\tilde F_i\in\mathcal F$ such that $P(\tilde F_i,\mathcal C_i)=\mu_i$. Moreover, we have 
    $$\mathcal Q_{-\Lambda}\subset \tilde F_i \cap E\subset \mathcal Q_\Lambda.$$
\end{lemma}
\begin{proof}
    This follows from the standard minimizing procedure. Take a minimizing sequence $F_l$ such that
    $$P(F_l,\mathcal C_i)\to \mu_i\mbox{ as }l\to\infty.$$
    We can make the following modifications. First, we can replace $F_l$ by $(F_l\cap \mathcal C_i)\cup(\bar F-\mathcal C_i)$ without changing the perimeter $P(\cdot,\mathcal C_i)$. Therefore, we may assume $F_l\Delta \bar F\subset \mathcal C_i$. Second, we can replace $F_l$ by $F_l':=(F_l\cup\mathcal Q_{-\Lambda})\cap \mathcal Q_\Lambda$, where the divergence theorem associated with Lemma \ref{Lem: barrier} shows
    $$P(F'_l,\mathcal C_i)\leq P(F_l,\mathcal C_i).$$
    This means that $F'_l$ is a minimizing sequence with $P(F_l',\mathcal C_i)\to \mu_i$ as $l\to\infty$. Since all $F'_l$ have locally uniformly bounded perimeter, we can assume $F'_l\to \tilde F_i$ as $l\to\infty$ in the sense of Caccioppoli sets for some $\tilde F_i\in \mathcal F$. From the lower semi-continuity of the perimeter, we have
    $$\mu_i\leq P(\tilde F_i,\mathcal C_i)\leq \mu_i,\mbox{ i.e. }P(\tilde F_i,\mathcal C_i)=\mu_i.$$
    From construction we have 
    $\mathcal Q_{-\Lambda}\subset F'_l\cap E\subset \mathcal Q_\Lambda$ and so the same thing holds for the limit $\tilde F_i$.
\end{proof}

\begin{proof}[Proof of Proposition \ref{Prop: minimizing hypersurface}]
    By Lemma \ref{Lem: minimizing exhaustion} $\partial\tilde F_i$ is $\mathbf T^k$-invariant and area-minimizing in $\mathcal C_i$. Then for any bounded smooth open subset $U$ we have 
    $$P(\tilde F_i,U)\leq \mathcal H^{N-1}(\partial U)\mbox{ for }i\mbox{ large enough}$$ from a simple comparison argument. In particular, we can deduce $\tilde F_i\to \tilde F_\infty$ as $i\to\infty$ in the sense of Caccioppoli sets up to a subsequence. Passing to the limit,  we have
    \begin{equation}\label{Eq: lie between}
        \mathcal Q_{-\Lambda}\subset \tilde F_\infty\cap E\subset \mathcal Q_\Lambda,
    \end{equation}
    and that $\partial\tilde F_\infty$ is $\mathbf T^k$-invariant and area-minimizing in $M$. Define $$\tilde\Sigma=\partial \tilde F_\infty.$$ 
    
    Let us verify all the desired properties for $\tilde \Sigma$. From the $\mathbf T^k$-invariance we can write
    $$\tilde F_\infty=\hat F_\infty\times \mathbf T^k\mbox{ and }\tilde \Sigma=\hat\Sigma\times \mathbf T^k\mbox{ with }\hat\Sigma=\partial\hat F_\infty.$$
     Note that the $\mathbf T^k$-invariant area-minimizing property of $\tilde \Sigma$ in $(M,g)$ is equivalent to the area-minimizing property of $\hat \Sigma$ in $(M_0,\hat g_0)$, where $M_0$ is the component in the product $M=M_0\times \mathbf T^k$ and $\hat g_0$ is the conformal metric given by
    $$\hat g_0=\left(\prod_{i=1}^ku_i\right)^{\frac{2}{n-1}}g_0.$$
Then it follows from \cite{NV20} that $\tilde\Sigma$ possibly has non-empty singular set $\mathcal S$, where $\mathcal S$ satisfies $\mathcal H^{N-8}(\mathcal S\cap K)<+\infty$ for any compact subset $K$.
    To show that $\tilde\Sigma$ is a smooth minimal graph around infinity, we shall use the notation $\hat P(\cdot,\cdot)$ to denote the $\hat g_0$-perimeters below. The area-minimizing property of $\hat \Sigma$ in $(M_0,\hat g_0)$ yields
    $$\hat P(\hat F_\infty,B^n_r)=O(r^{n-1})\mbox{ as }r\to+\infty,$$
    where $B^n_r$ denotes the coordinate $r$-ball given by $B^n_r=\{x\in E_0:|x|<r\}$ with $E_0$ the component in the product $E=E_0\times \mathbf T^k$.  In the end $E$, from \eqref{Eq: Schwarzschild type metric} we have
    $$g_0=v^{\frac{4}{N-2}}\mathrm dx\otimes \mathrm dx\mbox{ and }u_i=v^{\frac{2}{N-2}},$$
    and so
    $$\hat g_0=v^{\frac{4(N-1)}{(N-2)(n-1)}}\mathrm dx\otimes \mathrm dx.$$
    From the relation \eqref{Eq: lie between} we conclude that for any sequence of constants $\lambda_i\to 0$ we have
    $$\Phi_{\lambda_i}(\hat F_\infty\cap  E_0)\to \{x_n\leq 0\}\mbox{ in }\mathbf R^n_*$$
    as $i\to\infty$ in the sense of Caccioppoli sets up to a subsequence, where $\Psi_\lambda:\mathbf R^n\to \mathbf R^n,\,x\mapsto \lambda x$ and $\mathbf R^n_*=\mathbf R^n-\{O\}$. It follows from \cite[Theorem 36.3]{Sim83} that the convergence from $\partial\Phi_{\lambda_i}(\hat F_\infty\cap  E_0)$ to $\{x_n=0\}$ is locally smooth with multiplicity one in $\mathbf R^n_*$. In particular, we see that $\hat\Sigma$ can be written as a graph $(x',\hat w(x'))$ for some function $\hat w:\mathbf R^{n-1}\to [-\Lambda,\Lambda]$ around infinity, where we have 
    \begin{equation}\label{Eq: hat w decay}
        |D\hat w|+|x'||D^2\hat w|=o(1)\mbox{ as } x'\to\infty.
    \end{equation}
    Let us view $\hat \Sigma$ as a hypersurface in the Euclidean space $(\mathbf R^n,\mathrm dx\otimes \mathrm dx)$. Through a direct computation we know that the Euclidean mean curvature $\underline H$ of $\hat\Sigma$ satisfies
    \begin{equation*}
        |\underline H|=O(|\partial v|)=O(|x'|^{1-n})\mbox{ as }x'\to\infty.
    \end{equation*}
   Here and in the sequel, quantities with underline are those computed with respect to the Euclidean metric. By comparing the area of portions of $\Sigma$ to their domain up to some error on the boundary cylinder, we have
   \begin{equation*}
       \underline{\mathcal H}^{n-1}(\hat\Sigma\cap B^n_{|x|/2}(x))\leq \sigma_{n-1}\left(\frac{|x|}{2}\right)^{n-1}+O(|x|^{n-2})\mbox{ as }x\to+\infty,
   \end{equation*}
    where $\sigma_{n-1}$ denotes the volume of the $(n-1)$-dimensional unit Euclidean ball.
  From the Allard regularity theorem \cite[Theorem 24.2]{Sim83} as well as \eqref{Eq: lie between} we can deduce
  \begin{equation}\label{Eq: initial decay}
  |D \hat w|=O\left(|x'|^{-\gamma}\right)\mbox{ as }x'\to\infty\mbox{ with }\gamma=\frac{1}{4(n-1)}.
  \end{equation}
  Note that we have the mean curvature equation
  \begin{equation}\label{Eq: mean curvature equation}
      \Delta_{\mathbf R^{n-1}} \hat w=\underline H\sqrt{1+|D\hat w|^2}+\frac{\hat w_i\hat w_j\hat w_{ij}}{1+|D\hat w|^2}.
  \end{equation}
  We show that the decay of $D \hat w$ can be improved by a bootstrapping argument. By detailed computation we have 
  $$\underline H\sqrt{1+|D\hat w|^2}=(n-1)\Lambda v^{-\Lambda-1}(-\partial_jv D_j\hat w+\partial_nv)\mbox{ with }\Lambda=\frac{2(N-1)}{(N-2)(n-1)},$$
  where $D(\cdot)$ denotes the derivative in $\mathbf R^{n-1}$ and $\partial(\cdot)$ denotes the derivative in $\mathbf R^n$ evaluated at point $(x',\hat w(x'))$.
  In particular, we have
  $$\underline H\sqrt{1+|D\hat w|^2}=O(|x'|^{1-n})|D\hat w|+ O(|x'|^{-n})\mbox{ as }x'\to\infty.$$
  Combined with the estimates \eqref{Eq: lie between} and \eqref{Eq: hat w decay} we have $$\Delta_{\mathbf R^{n-1}}\hat w= O\left(|x'|^{\max\{1-n-\gamma,-n,-1-2\gamma\}}\right)=:F_1.$$
  By standard PDE theory the function $\hat w$ has a decomposition $\hat w=\hat w_1+\hat w_2$, where $\hat w_1$ is a bounded harmonic function with the expansion $c+\mathcal O_\infty(|x'|^{3-n})$ and $\hat w_2$ solves the equation $\Delta_{\mathbf R^{n-1}}\hat w_2=F_1$ with the estimate
  $$|\hat w_2|+|x'||D \hat w_2|= O(|x'|^{\max\{3-n-\gamma,2-n, 1-2\gamma\}}).$$
  In particular, we have the improved gradient estimate $$D \hat w=O(|x'|^{\max\{2-n,-2\gamma\}}).$$ Repeating the same argument for finite many times, we end up with the expansion
  $$\hat w=c_0+c_1|x'|^{3-n}+\mathcal O_1(|x'|^{2-n})\mbox{ as }x'\to \infty.$$
  Now let us write the mean curvature equation \eqref{Eq: mean curvature equation} as
  $$a_{ij}(D\hat w)\hat w_{ij}+b_j(v)\hat w_j=F_2,$$
  where
  $$a_{ij}(D\hat w)=\delta_{ij}-\frac{\hat w_i\hat w_j}{1+|D\hat w|^2},$$
  $$b_j=(n-1)\Lambda v^{-\Lambda-1}v_j\mbox{ and }F_2=(n-1)\Lambda v^{-\Lambda-1}\partial_n v.$$
  Let us introduce the notion $u= \mathcal O_{l,\alpha}(|x'|^{-\delta})$ to mean
  $$\sum_{j=0}^l|x'|^{j}|\partial^ju|+|x'|^{l+\alpha}|\partial^{l+\alpha}u|=O(|x'|^{-\delta}),$$
  where
  $$|\partial^{l+\alpha}u|(x'):=\sup_{x'_1,x_2'\in B^{n-1}_{|x|/2}(x)}\frac{|u(x'_1)-u(x'_2)|}{|x'_1-x'_2|}.$$
  Based on the relation $D_j=\partial_j+D_j\hat w\cdot\partial_n$,
  it is direct to verify
  $$a_{ij}=\delta_{ij}+\mathcal O_{l,\alpha}(1)=\delta_{ij}+O(|x'|^{-1}),\, b_j=\mathcal O_{l,\alpha}(|x'|^{1-n})\mbox{ and }  F_2=\mathcal O_{l,\alpha}(|x'|^{-n})$$
  if
  \begin{equation}\label{Eq: induction condition}
      D\hat w=\mathcal O_{l,\alpha}(1)\mbox{ and }D\hat w=O(|x'|^{-1}).
  \end{equation}
  Basically we use the condition on $D\hat w$ to show that each time we take derivative the decay order goes down by one. From \eqref{Eq: hat w decay} we see that \eqref{Eq: induction condition} holds for $l=0$ and any $\alpha\in (0,1)$. It follows
  from the Schauder estimate and bootstrapping that
  $$\hat w=c_0+c_1|x'|^{3-n}+\mathcal O_\infty(|x'|^{2-n})\mbox{ as }x'\to\infty.$$
  The graph function $w$ of $\tilde\Sigma$ is just the $\mathbf T^k$-invariant extension of the function $\hat w$, so it has the same expansion.

  We point out it suffices to show the strong stability of $\Sigma$ with respect to $\mathbf T^k$-invariant test functions. This simply follows from the facts
  $$\int_{\tilde\Sigma}(\Ric_g(\tilde \nu)+|\tilde A_{\tilde \Sigma}|)\phi^2\,\mathrm d\sigma_g=\int_{\tilde\Sigma}(\Ric_g(\tilde \nu)+|\tilde A_{\tilde \Sigma}|)\bar\phi^2\,\mathrm d\sigma_g$$
  and
  $$\int_{\tilde \Sigma}|\nabla_{\tilde \Sigma}\phi|^2\,\mathrm d\sigma_g\geq \int_{\tilde \Sigma}|\nabla_{\tilde \Sigma}\bar\phi|^2\,\mathrm d\sigma_g,$$
  where
  $$\bar\phi=\sqrt{\int_{\mathbf T^k}\phi^2\,\mathrm d\theta}.$$
  For $\mathbf T^k$-invariant test functions, the strong stability of $\tilde \Sigma$ is equivalent to the strong stability of $\hat \Sigma$ in the AF manifold $(M_0,\hat g_0)$, which was established by \cite[Lemma 3.6]{HSY2025}, where the singular set $\mathcal S$ does not matter since the test functions under consideration vanish around $\mathcal S$.
\end{proof}

\subsection{The torical symmetrization procedure}
In this subsection, we will give a proof for Proposition \ref{Prop: torical symmetrization}, which allows us to apply torical symmetrization inductively to an AF manifold with negative mass, up to dimension 13.
\begin{proposition}\label{Prop: torical symmetrization up to 13}
    Proposition \ref{Prop: torical symmetrization} is true up to dimension 13.
\end{proposition}
\begin{proof}
    Let $(M,g,E)$ be a $\mathbf T^k$-warped ALF manifold with nonnegative scalar curvature and negative mass. By Proposition \ref{prop: density} we may assume that the metric $g$ in $E$ has the form
    $$g=v^{\frac{4}{N-2}}\bar g\mbox{ where }v=1+A|x|^{2-n}+\mathcal O_\infty(|x|^{1-n})\mbox{ with }A<0.$$
    Moreover, the scalar curvature $R_g$ can be guaranteed to be positive in any fixed annulus $\mathcal A_{r,4r}$. It follows from Proposition \ref{Prop: minimizing hypersurface} that there is a strongly stable $\mathbf T^k$-invariant area-minimizing hypersurface $\Sigma=\Sigma_0\times \mathbf T^k$ possibly with singular set $\mathcal S$ satisfying $\mathcal H^{N-8}(\mathcal S\cap K)<+\infty$ for any compact subset $K$, which can be written as a graph over $\mathbf R^{n-1}\times \mathbf T^k$ around infinity, where the graph function $w$ satisfies
    \begin{equation}\label{Eq: graph expansion}
        w=c_0+c_1|x'|^{3-n}+\mathcal O_\infty(|x'|^{2-n}).
    \end{equation}
    From Remark \ref{Rem: stability of barrier} we can first determine the barrier constant $\Lambda$ and then fix an annulus $\mathcal A_{r,4r}$ with $r>\Lambda$, and so the area-minimizing hypersurface $\Sigma$ always has non-empty intersection with $\mathcal A_{r,4r}$, where the scalar curvature $R_g$ is positive.
    
    For convenience, we write the regular-singular decomposition of $\Sigma$ as 
    $$\Sigma=\mathcal R\sqcup\mathcal S.$$  Denote the induced metric of $\mathcal R$ from $(M,g)$ by $\check g$. Clearly, the metric $\check g$ is $\mathbf T^k$-warped and it can be written as
    $$\check g=\check g_{0}+\sum_{i=1}^k\check u_i^2\mathrm d\theta_i\otimes \mathrm d\theta_i,$$
    where $\check g_0$ is the induced metric of $\Sigma_0$ from $(M_0,g_0)$ on its regular part and $\check u_i=u_i|_{\mathcal R}$.
    From the expansion \eqref{Eq: graph expansion} we know that $(\Sigma,\check g)$ contains a $\mathbf T^k$-warped ALF end $\check E$ with zero mass around the infinity of $E$. Note that we can simply take $\check E$ to be the minimal graph part of $\Sigma$ around infinity with the coordinate $(x',\theta)$.
    
    Now we follow the work \cite{HSY26} to blow up the singular set $\mathcal S$ to make the regular part $\mathcal R$ complete. It follows from  Corollary \ref{Cor: singularity blowup} that we can find a $\mathbf T^k$-invariant positive harmonic function $G_{\mathcal S}$ on $\mathcal R$ with $G_{\mathcal S}=o(1)$ as $x'\to\infty$ such that the conformal metric
    $$\hat g_*=(1+\varepsilon G_{\mathcal S})^{\frac{4}{N-3}}\check g\mbox{ on }\mathcal R$$
    is complete for any positive constant $\varepsilon>0$. From the same argument of Lemma \ref{Lem: existence f eigenfunction} we can derive $G_{\mathcal S}=\mathcal O_2(|x'|^{3-n})$ and so $(\mathcal R,\hat g_*,\check E)$ is a $\mathbf T^k$-warped ALF manifold with

    \begin{equation}\label{Eq: small mass}
        m(\mathcal R,\hat g_*,\check E)=O(\varepsilon)\mbox{ as }\varepsilon\to 0.
    \end{equation}
    Set
    $$h=\frac{1}{2}(R_g+|A_\Sigma|^2)\mbox{ on }\mathcal R.$$
    Note that we have $h\geq 0$ in $\mathcal R$ and $h>0$ in $\mathcal R\cap \mathcal A_{r,4r}$. 
    It follows from the strong stability of $\Sigma$ and also \cite[Proposition 5.7]{HSY26} that we have
    
     \begin{equation}\label{eq:non negativity of scalar curvature	1}
    	 \int_{\mathcal R}|\nabla_{\hat g_*}\phi|^2+\frac{R_{\hat g_*}}{2}\phi^2\,\mathrm d\sigma_{\hat g_*}\geq \int_{\mathcal R}h_*\phi^2\,\mathrm d\sigma_{\hat g_*}\mbox{ with }h_*=h(1+\varepsilon G_{\mathcal S})^{-\frac{4}{N-2}}
     \end{equation}
    for any test function $\phi\in \mathcal T_{-q}(\mathcal R,\hat g_*,\check E)$ with $q>\frac{n-3}{2}$.
    Take a $\mathbf T^k$-invariant nonnegative cut-off function $\eta$ on $\Sigma$ with compact support in $\mathcal R\cap \mathcal A_{r,4r}$ such that $\eta> 0$ somewhere. 
    By \cite[Proposition 4.2]{HSY26} there is a smooth positive $\mathbf T^k$-invariant function on $\mathcal R$, denoted by $\hat u_{k+1}$, such that
    \begin{equation}\label{Eq: hat u k+1}
        -\Delta_{\hat g_*}\hat u_{k+1}+\frac{R_{\hat g_*}}{2}\hat u_{k+1}=\frac{h_*\eta}{2}\hat u_{k+1}\mbox{ and } \hat u_{k+1}=1+\mathcal O_2(|x'|^{3-n}).
    \end{equation}
    Define
    $$\widehat M=\mathcal R\times \mathbf S^1,\,\hat g=\hat g_*+\hat u_{k+1}^2\,\mathrm d\theta_{k+1}\otimes \mathrm d\theta_{k+1}\mbox{ and }\widehat E=\check E\times \mathbf S^1,$$
    where $\mathbf S^1$ denotes a unit circle and $\theta_{k+1}$ is the arc length parametrization of $\mathbf S^1$. It is direct to verify that $(\widehat M,\hat g,\widehat E)$ is a $\mathbf T^{k+1}$-warped ALF manifold with its scalar curvature given by
    $$R_{\hat g}=h_*\eta,$$
    which is nonnegative in $\mathcal R$ and positive in $\mathcal R\cap\mathcal A_{2r,3r}$. From \cite[Lemma 4.3]{HSY26} we have
    \begin{equation}\label{Eq: mass change}
        m(\widehat M,\hat g,\widehat E)\le m(\mathcal R,\hat g_*,\check E)-\frac{1}{2\omega_{n-1}}\int_{\check E}h_*\hat u_{k+1}^2\,\mathrm d\sigma_{\hat g_*}.
    \end{equation}
    As shown in \cite[(5.13)-(5.15)]{HSY26}, $\psi:=\hat u_{k+1}(1+\varepsilon G_{\mathcal S})$ is a smooth positive function independent of $\varepsilon$, which solves the equation
    $$-\Delta_{\check g}\psi+\frac{R_{\check g}}{2}\psi=\frac{h\eta}{2}\psi.$$
    Therefore, $\hat u_{k+1}$ is uniformly bounded from below by a positive constant in the support of $\eta$, and then we can guarantee $m(\widehat M,\hat g,\widehat E)<0$ by taking $\varepsilon$ small enough, concerning \eqref{Eq: small mass} and \eqref{Eq: mass change}.
\end{proof}

\section{Generic regularity and Riemannian positive mass theorem up to dimension 19}
\subsection{Generic regularity}
In this section, we are going to establish the following generic regularity result so that it can be used in our torical symmetrization procedure to relax the ambient dimension up to $19$.

\begin{proposition}\label{Prop: generic regularity}
    Let $(M^{n+k},g,E)$ be a $\mathbf T^k$-warped ALF manifold with
    \begin{equation}\label{Eq: 4.1}
        g=v^{\frac{4}{N-2}}\bar g \mbox{ in }E,
    \end{equation}
    where $v$ has the expansion
    \begin{equation}\label{Eq: 4.2}
        v=1+A|x|^{2-n}+\mathcal O_\infty(|x|^{1-n})\mbox{ with }A<0.
    \end{equation}
    Let $\mathfrak d=11$. Then, up to adjusting to a new metric with the same form above, there is a $\mathbf T^k$-invariant area-minimizing hypersurface $\tilde\Sigma$ satisfying that
    $\tilde\Sigma$ possibly has non-empty singular set $\mathcal S$ with $$\dim_{\mathcal H}\mathcal S<N-\mathfrak d,$$
    and also (ii)-(iii) in Proposition \ref{Prop: minimizing hypersurface}.
    Moreover, if $(M,g)$ has nonnegative scalar curvature, then we can guarantee that $\tilde \Sigma$ passing through a large annulus $\mathcal A_{r,4r}$ where the ambient scalar curvature is positive.
\end{proposition}

Let us fix some notations below. As before, we denote the reference set
$$\bar F=\{(x,\theta)\in E:x_{n-1}\leq 0\}$$
and define
$$\mathcal F=\{\mathbf T^k\mbox{-invariant Caccioppoli sets }F\mbox{ such that }(F\Delta \bar F)\cap E\mbox{ is bounded}\}.$$
Fix a sequence of constants $s_i\to+\infty$ as $i\to\infty$ and write $\mathcal C_i=\mathcal C_{s_i}$ for short.

Given any $\mathbf T^k$-warped ALF metric $g$ in the form of \eqref{Eq: 4.1}-\eqref{Eq: 4.2}, we introduce
$$\mu_i^g=\inf_{F\in\mathcal F}P_g(F,\mathcal C_i)$$
and
$$\mathfrak A_i^g=\{\partial \tilde F_i^g\cap \mathcal C_i:\tilde F^g_i\in \mathcal F\mbox{ with }P_g(\tilde F_i^g,\mathcal C_i)=\mu_i^g\},$$
where $P_g(F,\mathcal C_i)$ denotes the $g$-perimeter of $F$ in $\mathcal C_i$. Passing to a limit, we define
$$\mathfrak A^g=\{\partial \tilde F^g:\exists\,\partial\tilde F_i^g\in \mathfrak A^g_i\mbox{ such that }\partial\tilde F_i^g\to\partial \tilde F^g\}.$$
From the previous discussion in subsection \ref{Sec: existence minimizing}. In particular, we have $\mathfrak A^g\neq\emptyset$ and for each $\partial \tilde F^g\in \mathfrak A^g$ we have
$$\mathcal Q_{-\Lambda}\subset\tilde F^g\cap E\subset \mathcal Q_\Lambda\mbox{ for some }\Lambda>0.$$
Since $\partial \tilde F^g$ is area-minimizing, we have $\mathcal H^{N}(\partial \tilde F^g)=0$ and so we can always take $\tilde F^g$ to be an open subset of $M$. The same thing holds for $\tilde F^g_i$. We point out the following properties of $\tilde F^g$:
\begin{itemize}
    \item $\tilde F^g$ is connected. Otherwise, $\tilde F^g$ has a component $C$ such that $C\cap E$ is bounded. Then we can remove this component to decrease the area, contradicting to the area-minimizing property of $\partial \tilde F^g$.
    \item $\partial\tilde F^g$ has a unique component, denoted by $\partial_o\tilde F^g$, which is a smooth minimal graph around the infinity of $E$ by Proposition \ref{Prop: minimizing hypersurface}.
\end{itemize}
For convenience, we introduce
$$\mathfrak A_o^g=\{\partial_o\tilde F^g:\partial\tilde F^g\in \mathfrak A^g\}.$$
Roughly speaking, we will run the following two steps towards the generic regularity:
\begin{itemize}
    \item make a perturbation $g'$ to $g$ such that $\mathfrak A^{g'}_o$ consists of only one area-minimizing hypersurface;
    \item make a continuous family of perturbations $\{g'_t\}_{t\in [0,1]}$ to $g'$ such that area-minimizing hypersurfaces in $\mathfrak A^{g'_t}$ with $t\in [0,1]$ form a foliation (to be specified in Lemma \ref{Lem: foliation}).
\end{itemize}
All perturbations will be made in the $\mathbf T^k$-invariant setting.

\begin{lemma}\label{Lem: uniqueness}
    We can make an arbitrarily small $C^\infty$-perturbation $g'$ of $g$ such that $g'$ differs from $g$ only in $\mathcal A_{r,4r}$ and $\mathfrak A^{g'}_o$ consists only one area-minimizing hypersurface.
\end{lemma}
\begin{proof}
   Fix some $\partial\tilde F^g\in \mathfrak A^g$. Take a smooth $\mathbf T^k$-invariant function $w$ such that $w>0$ in $\mathring{\mathcal A}_{r,4r}-\partial \tilde F^g$ and $w=0$ elsewhere. Define
   $$g'=(1+\varepsilon w)^{\frac{1}{N-1}}g\mbox{ for some small }\varepsilon>0.$$
   We want to show $\mathfrak A^{g'}_o=\{\partial_o\tilde F^g\}$ when $\varepsilon$ is small enough. From Remark \ref{Rem: stability of barrier} we can still guarantee $\mathcal Q_{-\Lambda}\subset\partial \tilde F^{g'}\cap E\subset \mathcal Q_\Lambda$ for any $\partial\tilde F^{g'}\in\mathfrak A^{g'}$ when $\varepsilon$ is small enough. Then we see $\partial\tilde F^{g'}\cap\mathcal A_{2r,3r}\neq\emptyset$. It suffices to show 
   \begin{equation}\label{Eq: inclusion minimizing}
\partial\tilde F^{g'}\cap \mathcal A_{2r,3r}\subset \partial\tilde F^g\cap \mathcal A_{2r,3r}.
\end{equation}
   Then by unique continuation we have $\partial \tilde F^{g'}\subset \partial \tilde F^g$ and so $\partial_o\tilde F^{g'}=\partial_o\tilde F^g$.

   Now we illustrate how to prove \eqref{Eq: inclusion minimizing} for any $\partial \tilde F^{g'}\in \mathfrak A^{g'}$. Assume by contradiction that \eqref{Eq: inclusion minimizing} is false. Then we can fix a point $p\in \partial\tilde F^{g'}\cap\mathcal A_{2r,3r}$ with
   $$\dist_g(p,\partial\tilde F^g)\geq d>0\mbox{ where }d<r.$$
   By definition we can find $\partial\tilde F^g_i\in\mathfrak A^g_i$ and $\partial\tilde F^{g'}_i\in \mathfrak A^{g'}_i$ such that $\partial\tilde F^g_i\to\partial\tilde F^g$ and $\partial\tilde F^{g'}_i\to\partial \tilde F^{g'}$ as $i\to\infty$, where the convergence is also in the sense of Hausdorff distance in any fixed compact subsets. In particular, there are points $p_i\in\partial\tilde F^{g'}_i\cap\mathcal A_{2r,3r}$ such that
   $$\dist_g(p_i,\partial\tilde F^g_i)\geq \frac{d}{2}>0\mbox{ for }i\mbox{ large enough}.$$
   Let $B_i$ denote the $g$-geodesic ball centered at $p_i$ with radius $d/4$. Note that all $B_i$ have a definite amount of distance to $\{w=0\}$, so we have $w|_{B_i}\geq c_1>0$ for some constant $c_1$ independent of $i$. From the monotonicity formula we also have $\mathcal H_g^{N-1}(\partial\tilde F^{g'}_i\cap B_i)\geq c_2>0$ for some constant $c_2$ independent of $i$. Then we have
   \[
   \begin{split}
      \mathcal H^{N-1}_{g}(\partial\tilde F^{g'}_i\cap \mathcal C_i)+\varepsilon c_1c_2 &\leq\mathcal H^{N-1}_{g'}(\partial\tilde F^{g'}_i\cap\mathcal C_i)\\
      &\leq \mathcal H^{N-1}_{g'}(\partial\tilde F^g_i\cap\mathcal C_i)\\
      &\leq \mathcal H^{N-1}_{g}(\partial\tilde F^g_i\cap\mathcal C_i)+C\varepsilon|\nabla_g w|_{L^\infty}\cdot d_{\mathcal H,g}(\partial\tilde F^g_i,\partial\tilde F^g),
   \end{split}
   \]
   where $C$ denotes the area of $\partial\mathcal A_{r,4r}$ bounding $\mathcal H^{N-1}(\partial\tilde F^g_i\cap \mathcal A_{r,4r})$ and $d_{\mathcal H,g}(\cdot,\cdot)$ denotes the Hausdorff distance in $\mathcal A_{r,4r}$. As $i\to \infty$ we have $d_{\mathcal H,g}(\partial\tilde F^g_i,\partial\tilde F^g)\to 0$ and so
   $$\mathcal H^{N-1}_g(\partial\tilde F^{g'}_i\cap\mathcal C_i)<\mathcal H^{N-1}_g(\partial\tilde F^g_i\cap\mathcal C_i),$$
   which contradicts to the area-minimizing property of $\partial\tilde F^g_i$ in $\mathcal C_i$.
\end{proof}

In the following, we take $\partial \tilde F^{g'}_*\in \mathfrak A^{g'}$ where $\tilde F^{g'}_*$ is {\it innermost} among $\mathfrak A^{g'}$ and denote
$$\Sigma_o=\partial_o\tilde F^{g'}_*.$$
Recall that we have $\mathfrak A^{g'}_o=\{\Sigma_o\}$ from Lemma \ref{Lem: uniqueness} above.
\begin{lemma}\label{Lem: foliation}
   Let $\Sigma_o$ be as above.
   Then we can construct a continuous family of arbitrarily small $C^\infty$-perturbations $\{ g'_t\}_{t\in [0,1]}$ of $g'$ with $ g'_0=g'$ satisfying the following properties:
    
    \begin{itemize}
        \item[(1)] $g'_t$ differs from $g'$ only in $\mathcal A_{2r,3r}$ for all $t\in [0,1]$; 
        \item[(2)] all distinct area-minimizing hypersurfaces $\Sigma_t\in \mathfrak A^{g'_{t}}_o$ are pairwise disjoint;
        \item[(3)] we have $\tilde F^{g'}_*\subset \tilde F^{g'_t}$ for any $\partial\tilde F^{g'_t}\in \mathfrak A^{g'_t}$;
        \item[(4)] there is a positive continuous function $c$ on $M$ such that for any pair of area-minimizing hypersurface $\Sigma_{t_i}\in\mathfrak A^{g'_{t_i}}_o$ we have
        $$\dist_{g'}(p_1,p_2)\geq \min\{ c(p_1),c(p_2)\}\cdot|t_1-t_2|\mbox{ for any }p_i\in\Sigma_{t_i}.$$
    \end{itemize}
     
\end{lemma}
\begin{proof}
    The proof is almost the same as that of \cite[Claim 9.3]{CMSW2025}. Take a smooth piece $\mathcal R$ of $\Sigma_o\cap \mathcal A_{2r,3r}$ with a fixed Fermi neighborhood $V$ in $\mathcal A_{2r,3r}$ such that
    $$\partial \tilde F^{g'}\cap V=\Sigma_o\cap V\mbox{ for all }\partial \tilde F^{g'}\in \mathfrak A^{g'}.$$
     Fix a smooth $\mathbf T^k$-invariant function $w$ supported in the Fermi neighborhood $V$ such that $w\equiv 0$ along $\mathcal R$, $w<0$ in the interior of $V'-\tilde F^{g'}_*$ and also $w>0$ in $V'\cap \tilde F^{g'}_*$, where $V'$ is a smaller Fermi neighborhood of a smaller piece $\mathcal R'$ of $\mathcal R$. Moreover, we can require $\partial_zw<0$ in $V'$, where $z$ is the vertical Fermi coordinate with $\partial_z$ pointing outside $\tilde F^{g'}_*$. 
    
    Define the metrics $g'_t$ by
    $$g'_t=e^{2tw}g'\mbox{ with }t\in [0,1].$$
    Clearly, we have the property (1). Note that we can always replace $w$ by a smaller multiple of $w$, and then the fact $\mathfrak A^{g'}_o=\{\Sigma_o\}$ implies
    \begin{itemize}
        \item $\partial \tilde F^{g'_t}\cap V=\partial_o\tilde F^{g'_t}\cap V$ for all $\partial \tilde F^{g'_t}\in \mathfrak A^{g'_t}$;
        \item $ \Sigma_t\cap V$ can be written as a graph over $\mathcal R$ for all $\Sigma_t\in\mathfrak A^{g'_t}_o$, where the graph functions are almost constant.
    \end{itemize}
    Take two distinct area-minimizing hypersurfaces
    $$\Sigma_{t_1}\in \mathfrak A^{g'_{t_1}}_o\mbox{ and }\Sigma_{t_2}\in \mathfrak A^{g'_{t_2}}_o.$$
   If 
   $t_1=t_2$,
   we assert that $\Sigma_{t_1}$ and $\Sigma_{t_2}$ are disjoint. Otherwise, by the unique continuation $\Sigma_{t_1}$ and $\Sigma_{t_2}$ must intersect transversally somewhere. By definition, for $j=1,2$ we have
    $$\Sigma_{t_j}=\partial_o\tilde F^{g'_{t_j}}\mbox{ and }\partial \tilde F^{g'_{t_j}}_i\to \partial\tilde F^{g'_{t_j}}\mbox{ as }i\to \infty\mbox{ with }\partial \tilde F^{g'_{t_j}}_i\in \mathfrak A^{g'_{t_j}}_i.$$
    Then we know that $\partial \tilde F^{g'_{t_1}}_i$ actually intersects $\partial \tilde F^{g'_{t_2}}_i$ transversally for $i$ large enough, which is impossible from a cut-and-paste argument.
    Now let us assume $t_1>t_2$. With a careful choice of suitable function $w$, the same argument as that in the proof of \cite[Claim 9.3]{CMSW2025} actually implies 
    \begin{itemize}
        \item $\partial(\tilde F^{g'_{t_1}}_i\cup \tilde F^{g'_{t_2}}_i)\in \mathfrak A^{g'_{t_1}}_i$ for $i$ large enough;
        \item $\partial\tilde F^{g'_{t_2}}_i$ is weakly mean-concave everywhere and strictly mean-concave in $V$ with respect to the outward unit normal and the metric $g'_{t_1}$ for $i$ large enough, and the same thing holds for $\partial\tilde F^{g'_{t_2}}$.
    \end{itemize}
    From these two facts we conclude that $\Sigma_{t_1}$ lies strictly above $\Sigma_{t_2}$ for any $t_1>t_2$, and so we obtain the property (2). We point out that the same argument above also implies 
    $$\partial(\tilde F^{^{g'_t}}\cap \tilde F^{g'}_*)\in \mathfrak A^{g'}.$$
    Since $\tilde F^{g'}_*$ is innermost, we obtain the property (3). Finally,
    the property (4) follows from an exhaustion argument and \cite[(9.21)]{CMSW2025}.
\end{proof}
        
\begin{corollary}\label{Cor:generic regularity}   
For almost every $t\in [0,1]$ we have
    $$\dim_{\mathcal H}(\mathcal S)<N-11\mbox{ for any }\Sigma_t\in \mathfrak A^{g'_t}_o,$$
    where $\mathcal S$ denotes the singular set of $\Sigma_t$.
\end{corollary}
\begin{proof}
Decompose $\partial\tilde F^{g'}_*=\Sigma_o\sqcup \Sigma$ and denote $\bar M=M-\Sigma$. Then all $\Sigma_t\in \mathfrak A^{g'_t}_o$ are area-minimizing boundaries in $\bar M$. Combined with Lemma \ref{Lem: foliation} the desired consequence follows from \cite[Theorem 1.4 and Corollary 1.6]{CMSW2025} applied to the triple
    $$\left(\bar M,\bigcup_{t\in [0,1]}\mathfrak A^{g'_t}_o,\mathfrak T\right),$$
    where $\mathfrak T$ is the function given by
    $$\mathfrak T:\bigcup_{\Sigma_t\in \mathfrak A^{g'_t}_o}\Sigma_t\to [0,1]\mbox{ with }\mathfrak T(p)=t\mbox{ if }p\in \Sigma_t.$$
    We point out that \cite[Theorem 1.4]{CMSW2025} is valid with the locally Lipschitz property of $\mathfrak T$-function by a simple exhaustion argument.
\end{proof}

We are ready to prove Proposition \ref{Prop: generic regularity}.
\begin{proof}[Proof of Proposition \ref{Prop: generic regularity}]
    Let $g$ be a $\mathbf T^k$-warped ALF matric satisfying \eqref{Eq: 4.1}-\eqref{Eq: 4.2}. From Proposition \ref{prop: density} we can make a perturbation $g_*$ of $g$ such that $g_*$ still satisfies \eqref{Eq: 4.1}-\eqref{Eq: 4.2} and that the scalar curvature $R_{g_*}$  becomes positive in $\mathcal A_{r,4r}$ for some fixed large annulus $\mathcal A_{r,4r}$ to be determined later. From Remark \ref{Rem: stability of barrier} we can assume the barrier constant $\Lambda$ from \ref{Lem: barrier} to be fixed under small perturbations no matter which annulus is chosen. 
    Now we can apply 
    Corollary \ref{Cor:generic regularity}     to perturb $g_*$ into a new metric $\tilde g$ such that any area-minimizing hypersurface $\tilde \Sigma\in \mathfrak A^{\tilde g}_o$ satisfies $\dim_{\mathcal H}(\mathcal S)<N-11$ Since the perturbation is made in compact subset, we can assume the barrier constant $\Lambda$ is still kept and so $\tilde \Sigma$ passes through the region $\mathcal A_{r,4r}$ where the scalar curvature $R_{\tilde g}$ is still positive. The properties (ii)-(iii) follow from the exactly same argument as before.
\end{proof}

\begin{lemma}\label{Lem: modified torical symmetrization}
    If Proposition \ref{Prop: generic regularity} holds for some $\mathfrak d\in \mathbf N_+$, then Proposition \ref{Prop: torical symmetrization} still holds when $k+3<\dim M\leq 2\mathfrak d-3$. 
\end{lemma}
\begin{proof}
    We follow the proof of Proposition \ref{Prop: torical symmetrization up to 13} with Proposition \ref{Prop: generic regularity} replaced by Proposition \ref{Prop: minimizing hypersurface}. 
\end{proof}

Now we have Proposition \ref{Prop: torical symmetrization} as an immediate consequence.
\begin{proof}[Proof of Proposition \ref{Prop: torical symmetrization}]
    This follows directly from Proposition \ref{Prop: generic regularity} and Lemma \ref{Lem: modified torical symmetrization}.
\end{proof}
\subsection{Proof of Riemannian positive mass theorem up to dimension 19}
In this subsection, we complete the proof of the Riemannian positive mass theorem up to dimension $19$ based on the torical symmetrization procedure. 

As a fundamental tool, we mention that the positive mass theorem for $\mathbf T^k$-warped ALF manifolds holds in all dimensions with enough $\mathbf T^k$-symmetry, which was first established in \cite[Proposition 2.9]{HSY26}. Although the proof there is only for $\mathbf T^1$-warped case, it can be extended to the $\mathbf T^k$-warped case without change. Actually, their proof says that the Lohkamp compactification and the soap bubble method used in \cite{CLSZ25} to prove the positive mass theorem for ALF manifolds run smoothly within the framework of $\mathbf T^k$-invariant objects, where the singularity issue never happens because of the fact that singularity has to be $\mathbf T^k$-invariant with codimension no less than $7$.

\begin{proposition}\label{Prop: PMT ALF}
    Let $(M,g,E)$ be a $\mathbf T^k$-warped ALF manifold with nonnegative scalar curvature and $3\leq \dim M-k\leq 7$. Then we have $m(M,g,E)\geq 0$ and the equality holds if and only if $(M,g)$ is isometric to $(\mathbf R^n\times \mathbf T^k,\bar g)$, where $\bar g$ is the reference metric from Definition \ref{def: ALF}.
\end{proposition}

We are ready to prove Theorem \ref{Thm: PMT}.
\begin{proof}[Proof of Theorem \ref{Thm: PMT}]
It suffices to prove $m(M,g,E)\geq 0$, since the rigidity with mass vanishing follows from a standard deformation argument combined with the Bishop-Gromov volume comparison theorem. Assume $m(M,g,E)<0$ by contradiction. It follows from Proposition \ref{Prop: torical symmetrization} that we can construct a $\mathbf T^k$-warped ALF manifold $(\tilde M,\tilde g,\tilde E)$ with nonnegative scalar curvature and negative mass, where $\dim \tilde M=k+3$. This contradicts to Proposition \ref{Prop: PMT ALF}.
\end{proof}

\section{Geroch conjecture in dimension 12}\label{Sec: Geroch}

In this section, we present a proof of Theorem \ref{thm:12 dim Geroch}, where $X$ can be assumed to be orientable up to lifting. We will adopt the dimension reduction argument by Schoen-Yau \cite{SY79b}. We use $M^n$ to denote $\mathbf{T}^n\# X^n$. Using the generic regularity theorem in \cite{CMSW2025}, we can find a suitable homological minimizing hypersurface $\Sigma$ in $M^n = \mathbf{T}^n\# X^n$ with singular set $\mathcal{S}$, such that the Hausdorff dimension is less than $1$. Then, we prove $\Sigma\backslash\mathcal{S}$ is an open SYS manifold as introduced by \cite{SWWZ24}, thereby admitting no positive scalar curvature (PSC) metric, see Lemma \ref{lem: degree calculation} and Lemma \ref{lem: open SYS verification}. Finally, we blow up the singular set $\mathcal{S}$ using the method developed in Section 2.1 and obtain a contradiction.

The following lemma reveals certain topological non-triviality for singular representatives in homology classes.

\begin{lemma}\label{lem: degree calculation}
   Let $f:M^{n}\longrightarrow \mathbf{T}^n = \mathbf{T}^{n-1}\times\mathbf{S}^1$ be a $C^1$ map of non-zero degree that is transversal to $\mathbf{T}^{n-1}\times\{1\}$, so $\Sigma_0 = f^{-1}(\mathbf{T}^{n-1}\times\{1\})$ is a regular submanifold.  Let $\Sigma$ be a homological minimizing hypersurface with singular set $\mathcal{S}$ in the homology class $[\Sigma_0]\in H_{n-1}(M^n)$. Then there exists a map $F:\Sigma\longrightarrow\mathbf{T}^{n-1}$, such that the restriction $F|_{\Sigma\backslash\mathcal{S}}:\Sigma\backslash\mathcal{S}\longrightarrow \mathbf{T}^{n-1}\backslash F(\mathcal{S}) $ is a proper map of non-zero degree.
\end{lemma}
\begin{proof}
     By classical results in geometric measure theory we can find a homological minimizing integer multiplicity $(n-1)$-current $\tau$ with $\partial R = \tau-\|\Sigma_0\|$, where $R$ is an integer multiplicity $n$-current. 
     
    Consider the composition of the following sequence of maps
    \begin{align*}
        \Phi: M^n\stackrel{f}{\longrightarrow} \mathbf{T}^n = \mathbf{T}^{n-1}\times \mathbf{S}^{1}\longrightarrow \mathbf{T}^{n-1}\times \{1\} = \mathbf{T}^{n-1},
    \end{align*}
    where the last map is the projection map. Let $F = \Phi\circ i: \Sigma\longrightarrow \mathbf{T}^{n-1}$, where $i$ is the inculsion map of $\Sigma$ in $M$. We first note that $ \deg (\Phi|_{\Sigma_0}) = \deg f \ne 0$. To see this, pick a regular value $x\in \mathbf{T}^{n-1}$ of the map $\Phi|_{\Sigma_0}$. Since we have assumed $f$ is transverse to  $\mathbf{T}^{n-1}\times\{1\}$ in our statement of the lemma,  $x$ is also a regular value of $f$. Hence
    \begin{align*}
        \deg \Phi|_{\Sigma_0} = \#\{(\Phi|_{\Sigma_0})^{-1}(x)\} = \#\{f^{-1}(x)\} = \deg f \ne 0.
    \end{align*}
    
    Let $\omega\in\Omega_c^{n-1}(\mathbf{T}^{n-1}\backslash F(\mathcal{S}))$ be a differential form with compact support and integration $1$ on $\mathbf{T}^{n-1}\backslash F(\mathcal{S})$. We have
    \begin{align*}
        \deg(F|_{\Sigma\backslash \mathcal{S}}) =& \int_{\Sigma\backslash\mathcal{S}}F^*\omega =  \int_{\Sigma\backslash\mathcal{S}}i^* \Phi^*\omega = \tau(\Phi^*\omega) = \|\Sigma_0\|(\Phi^*\omega)\\
        = &\int_{\Sigma_0}\Phi^*\omega = \deg (\Phi|_{\Sigma_0}) = \deg f \ne 0.
    \end{align*}
    Thus $\deg(F|_{\Sigma\backslash \mathcal{S}})\ne 0$ concluding the proof.
\end{proof}

Now, let us recall the definition of the open SYS manifold in \cite{SWWZ24}. 

\begin{definition}\label{def: aspherical homology class}(\cite[Definition 1.1]{SWWZ24})
    Given an orientable open manifold $M$, a homology class $\tau\in H_2(M,\partial_\infty M)$ is said to be \textit{aspherical}, if there exists an open set $\Omega$ with compact closure, such that the restricted class $\tau|_\Omega\in H_2(M,M\backslash\Omega)$ is \textit{aspherical}, $i.e.$ $\tau|_\Omega$ does not lie in the Hurewicz image of $\pi_2(M,M\backslash\Omega)$.
\end{definition}
In the above definition, $H_k(M,\partial_\infty M)$ is the $k$-th locally finite homology group given by

\[
H_k(M,\partial_\infty M) = H_k^{\mathrm{lf}}(M)
:=
H_k(\varprojlim_{K \Subset M}C_*(M, M \setminus K)),
\]
where the inverse limit is taken over compact subsets 
$K \subset M$ ordered by inclusion. 

\begin{remark}\label{rmk: restriction map}

$\quad$

\begin{enumerate}
    \item $H_k(M,\partial_\infty M)$ and $H_k^{\mathrm{lf}}(M)$ represent the same object.We use the former notation when convenient, in order to align with the notation in \cite{SWWZ24}.
    \item Let $i_\Omega: \Omega\longrightarrow M$ be the inclusion map. There are two equivalent ways of defining the restricted map  $\tau \mapsto \tau|_\Omega$ in Definition \ref{def: aspherical homology class}. One is described in \cite[p.5]{SWWZ24}, given by
    \begin{align*}
        H_k^{\mathrm{lf}}(M)\longrightarrow H_k(M,M\backslash\Omega)
    \end{align*}
    via the universal property of the inverse limit.  A definition of more geometric flavor can be found in \cite[Section 6.1.3]{Morgan}, which gives rise a map
    \begin{align*}
        (i_\Omega)_{\#}: H_k^{\mathrm{lf}}(M)\longrightarrow  H_k^{\mathrm{lf}}(\Omega).
    \end{align*}
    When $\Omega\subset\subset M$ has smooth boundary, it follows that $ H_k^{\mathrm{lf}}(\Omega)\cong H_k(\Omega,\partial\Omega)\cong H_k(M,M\backslash\Omega)$, so $\tau|_\Omega$ and $(i_\Omega)_{\#}(\tau)$ can be identified as the same object.
\end{enumerate}
\end{remark}

\begin{remark}\label{rmk: Omega and Omega'}
Let $\Omega'\supset\Omega$ be any neighborhood of $\Omega$, then the commutative diagram
    \[
\xymatrix{
\pi_2(M, M\setminus \Omega') 
  \ar[r] 
  \ar[d]_{h_{\mathrm{rel}}} 
&
\pi_2(M, M\setminus \Omega) 
  \ar[d]^{h_{\mathrm{rel}}} 
\\
H_2(M, M\setminus \Omega') 
  \ar[r]
&
H_2(M, M\setminus \Omega)
}
\]
implies $\tau|_{\Omega'}$ is aspherical as long as $\tau|_\Omega$ is aspherical.
\end{remark}

\begin{definition}\label{def: open SYS}
    An orientable manifold $M^n$ is said to be open SYS, if there exists $\beta_1,\beta_2,\dots,\beta_{n-2}\in H^1(M)$, such that
    \begin{align*}
        [M,\partial_\infty M]\smallfrown\beta_1\smallfrown\beta_2\smallfrown\dots\smallfrown\beta_{n-2}\in H_2(M,\partial_\infty M) 
    \end{align*}
    is aspherical. Here $[M,\partial_\infty M]\in H_n(M,\partial_\infty M)$ represents the fundamental class of the locally finite homology.
\end{definition}

Definition \ref{def: open SYS} is slightly more restrictive than \cite[Definition 1.3]{SWWZ24}; the definition here corresponds to the special case $\partial_b M = \emptyset$ in the terminology of \cite{SWWZ24}. In this important sub-case, \cite{SWWZ24} treated the PSC obstruction problem by employing a dimension reduction argument via compact free-boundary weighted slicing on a sufficiently large compact set $\Omega_T\subset M$ when $n\le 7$. By combining their approach with \cite[Theorem 1.2]{CMSW2025}, we extend the result to dimensions up to $n \le 11$.

\begin{lemma}\label{open SYS dim 11}
    Let $M^n$ be an open SYS manifold as defined in Definition \ref{def: open SYS} $(n\le 11)$. Then $M^n$ admits no $\mathbf{T}^k$-warped complete PSC metric, $i.e.$ there exists no complete metric $g$ and smooth functions $u_1,u_2,\dots, u_k$ on $M$, such that the $\mathbf{T^k}$-warped manifold
    \begin{align*}
        (M\times\mathbf{T}^k,g+\sum_{i=1}^k u_i^2d\theta_i^2)
    \end{align*}
    has positive scalar curvature.
\end{lemma}

\begin{proof}

    We first consider the case $k=0$. Our argument follows the proof of \cite[Theorem 1.5]{SWWZ24}. Assume that $M$ carries a metric $g$ satisfying $R_g>0$. Let $\tau = [M,\partial_\infty M]\smallfrown\beta_1\smallfrown\beta_2\smallfrown\dots\smallfrown\beta_{n-2}\in H_2(M,\partial_\infty M)$. By Definition \ref{def: open SYS}, $\tau$ is aspherical. Thus, there exists an open set $\Omega$ with compact closure such that $\tau|_{\Omega}\in H_2(M,M\backslash \Omega)$ does not lie in the Hurewicz image $h_{\mathrm{rel}}: \pi_2(M,M\backslash\Omega)\longrightarrow H_2(M,M\backslash\Omega)$.
    
Consider the mollification of the distance function, we construct the following terms
    \begin{itemize}
        \item $\rho:M\longrightarrow [0,+\infty)$ a smooth function with $\rho|_\Omega = 0$ and $\Lip\rho<1$;
        \item $\Omega_s = \rho^{-1}([0,s])$ for any $s\ge 0$;
        \item $L(s):= \min_{\Omega_s}R_g$.
    \end{itemize}

    We fix $T>T_0 = T_0(L,s_0)$, where $T_0(L,s_0)$ is the constant given by \cite[Proposition 3.2]{SWWZ24}. Then from Remark \ref{rmk: Omega and Omega'} we know $\tau|_{\Omega_T}$ is not contained in the image of the Hurewicz map $h_{\mathrm{rel}}: \pi_2(M,M\backslash\Omega_T)\longrightarrow H_2(M,M\backslash\Omega_T)\cong H_2(\Omega_T,\partial\Omega_T) $.

    In the same spirit of \cite[p.23, Step 1]{SWWZ24}, we aim to find a sequence of weighted slicing
    \begin{align*}
        (\Sigma_2,\partial\Sigma_2,w_2,g_2)\hookrightarrow (\Sigma_3,\partial\Sigma_3,w_3,g_3)\hookrightarrow\dots\hookrightarrow (\Sigma_n,\partial\Sigma_n,w_n,g_n) = (\Omega_T,\partial\Omega_T,1,g).
    \end{align*}
    Here, $(\Sigma_j,\partial\Sigma_j)$ represents the homology class $[\Omega_T,\partial\Omega_T]\smallfrown e^*\beta_1\smallfrown\dots\smallfrown e^*\beta_{n-j}\in H_j(\Omega_T,\partial\Omega_T)$, where $e: \Omega_T\longrightarrow M$ is the inclusion map. Moreover, $w_j$ is a smooth function on $\Sigma_j$, and $\Sigma_{j-1}$ minimizes the weighted area
    \begin{align*}
        \mathcal{A}_j(\Sigma) = \int_\Sigma w_j d\mathcal{H}^{j-1}_{g_j}
    \end{align*}
    in $\Sigma_j$, and $q_j = \frac{w_{j-1}}{w_j|_{\Sigma_{j-1}}}$ is a first eigenfunction of the stability operator on $\Sigma_{j-1}$ associated to the $w_j$-weighted area.

    We now discuss the construction in detail. At the $(n-j+1)$-st step, we aim to find a $w_j$-weighted minimizer $\Sigma_{j-1}$ inside $\Sigma_j$. Initially, $\Sigma_j$ is equipped with the induced metric $g_{j+1}\big|_{\Sigma_j}$ inherited from $(\Sigma_{j+1}, g_{j+1})$.

Compared with \cite{SWWZ24}, an additional issue arises in our setting: for $8 \le j \le 11$, we must perturb $(\Sigma_j, g_{j+1}\big|_{\Sigma_j})$ to a new metric $(\Sigma_j, g_j)$ by applying the generic regularity result of \cite{CMSW2025}, in order to ensure that $\Sigma_{j-1}$ is smooth. This procedure gives rise to an additional data $g_j$ in our weighted slicing. 

At each step, we may also remove an arbitrarily small neighborhood of $\Sigma_{j-1}$ in order to eliminate possible boundary singularities. $i.e.$, we may throw arbitrarily small neighborhood of $\Sigma_{j-1}$ away at each step to get rid of boundary singularities. Then we solve $q_j\in C^\infty(\Sigma_{j-1})$ as the Dirichlet first eigenfunction of the stability operator on $\Sigma_{j-1}$ associated to the $w_j$-weighted area, and define
\begin{align*}
    w_{j-1} = q_j \bigl(w_j\big|_{\Sigma_{j-1}}\bigr).
\end{align*}
This completes the construction of the weighted slicing.

Since both the metric perturbation and the removal of the boundary neighborhood can be made arbitrarily small at each step, the following condition can be arranged for every $\epsilon > 0$:
    \begin{align*}
        & \|g_2-(g|_{\Sigma_2})\|_{C^2}<\epsilon,\\
        &R(g_2 +\sum_{j=2}^{n-1}q_j^2d\theta_j^2)\ge  R_g|_{\Sigma_2}-\epsilon>0,\\
        &d_{(\Sigma_2,g_2)}(x,y) \ge (1-\epsilon)d_{(\Omega_T,g)}(x,y),\mbox{ for all }x,y\in \Sigma_2.
    \end{align*}
    Then, we may follow \cite[p.24, Step 2]{SWWZ24} which applies \cite[Proposition 3.2]{SWWZ24} to get a contradiction.

    For general $k$, the only difference from the $k=0$ case is that we select
    \begin{align*}
        (\Omega_T,\partial\Omega_T,u_1u_2\dots u_k,g)
    \end{align*}
    as the first slice in the weighted slicing procedure. The remaining part goes in the same way as the $k=0$ case.
\end{proof}

In \cite{SWWZ24}, it was proved that $\mathbf{T}^n \setminus \Gamma$ is open SYS when $\Gamma$ is a disjoint union of closed curves. Our goal is to extend this result to the case where $\dim_{\mathcal{H}}\Gamma < 1$ (see Lemma~\ref{lem: open SYS verification} below). To this end, we first record the following simple observation.

\begin{lemma}\label{lem: no intersection}
    Let $\Gamma\subset \mathbf{T}^n$ be a compact set with $\dim_{\mathcal{H}}\Gamma<1$. Let $\mathbf{T}^n = \mathbf{T}^{n-1}\times \mathbf{S}^1$ be a product decomposition of $\mathbf{T}^n$. Then  $\Gamma$ has a neighborhood $N$ in $\mathbf{T}^n$, such that $\bar{N}\cap( \mathbf{T}^{n-1}\times\{\theta_0\}) = \emptyset$ for some $\theta_0\in \mathbf{S}^1$.
\end{lemma}

\begin{proof}
    By coarea formula, we have
    \begin{align*}
        0 = \mathcal{H}^1(\mathcal{S}) = \int_{\mathbf{S}^1}\mathcal{H}^0(\mathcal{S}\cap(\mathbf{T}^{n-1}\times\{\theta\}))d\theta
    \end{align*}
    Thus, there exists $\theta_0\in\mathbf{S}^1$, such that $\mathcal{H}^0(\mathcal{S}\cap (\mathbf{T}^{n-1}\times\{\theta_0\})) = 0$, which implies $\mathcal{S}\cap(\mathbf{T}^{n-1}\times\{\theta_0\})$ is an empty set. Since $\mathcal{S}$ is closed, we can find an open neighborhood $N$ of $\mathcal{S}$ such that $\bar{N}\cap (\mathbf{T}^{n-1}\times \{\theta_0\}) = \emptyset$.
\end{proof}

\begin{lemma}\label{lem: open SYS verification}
    Let $\Gamma\subset \mathbf{T}^n$ be a compact set with $\dim_{\mathcal{H}}\Gamma<1$ and let $Y = \mathbf{T}^n\setminus\Gamma$. Then any oriented open manifold $Z$ which admits a proper map $f_Z:Z\longrightarrow Y$ with non-zero degree is an open SYS manifold.
\end{lemma}

\begin{proof}

         Assume $\mathbf{T}^n$ has the form $\mathbf{T}^{n-1}\times \mathbf{S}^1$ and $N$ is constructed as in Lemma \ref{lem: no intersection}. Let $\alpha_1,\alpha_2,\dots,\alpha_{n-1}\in H^1(\mathbf{T}^n)$ be the cohomology classes represented by the first $n-1$  $\mathbf{S}^1$ factors in $\mathbf{T}^n = \mathbf{T}^{n-1}\times \mathbf{S}^1$, so 
    \begin{align*}
        \tau_0 = [\mathbf{T}^n]\smallfrown \alpha_1\smallfrown \alpha_2\smallfrown\dots\smallfrown\alpha_{n-1}\in H_1(\mathbf{T}^n)
    \end{align*}
    is the homology class represented by the last $\mathbf{S}^1$ factor in  $\mathbf{T}^n = \mathbf{T}^{n-1}\times \mathbf{S}^1$. 

    Let $j: Y\longrightarrow \mathbf{T}^n$ be the inclusion map. Since $Y$ is an open set of $\mathbf{T}^n$, there is a restriction map $j_{\#}: H_k(\mathbf{T}^n)\longrightarrow H_k^{\mathrm{lf}}(Y)$\footnote{We slightly change the notation in the proof of Lemma \ref{lem: open SYS verification} for better exposition: For an open oriented $n$-dimensional manifold $M$, we use $H_k^{\mathrm{lf}}(M)$ to denote its $k$-th locally finite homology group, which is equivalent to $H_k(M,\partial_\infty M)$ used before. We use $[M]\in H_n^{\mathrm{lf}}(M)$ to denote the fundamental class.} in the sense of Remark \ref{rmk: restriction map} (2). By naturality of the cap product and compatibility of fundamental classes under restriction, the Poincare duality isomorphisms for $\mathbf{T}^n$ and $Y$ fit into the following commutative diagram.
    \[
\xymatrix{
H^k(\mathbf{T}^n) \ar[r]^{\;\cap [\mathbf{T}^n]\;} \ar[d]_{j^*}
& H_{n-k}(\mathbf{T}^n) \ar[d]^{j_{\#}} \\
H^k(Y) \ar[r]^{\;\cap [Y]\;}
& H_{n-k}^{\mathrm{lf}}(Y)
}
\]
Denote $\beta_k = j^*\alpha_k\in H^1(Y)$ $ (k=1,2\dots,n-1)$, we have
    \begin{align*}
        \tau_Y& = [Y]\smallfrown \beta_1\smallfrown \beta_2\smallfrown\dots\smallfrown \beta_{n-1}\\
        & = [Y]\smallfrown j^*\alpha_1\smallfrown j^*\alpha_2\smallfrown\dots\smallfrown j^*\alpha_{n-1}\\
        & = j_{\#}([\mathbf{T}^n]\smallfrown \alpha_1\smallfrown \alpha_2\smallfrown\dots\smallfrown \alpha_{n-1}) = j_{\#}\tau_0\in H_1^{\mathrm{lf}}(Y),
    \end{align*}

Let $i_N: H_*^{\mathrm{lf}}(Y)\longrightarrow H_*(\mathbf{T}^n, N)$ be the restriction map in the sense of Remark \ref{rmk: restriction map} (2) and $i: \mathbf{T}^n\longrightarrow (\mathbf{T}^n,N)$ be the inclusion map. We have the following commutative diagram\footnote{Note $ H_*(\mathbf{T}^n, N)\cong H_*(\mathbf{T}^n\backslash N,\partial N)\cong H_*^{\mathrm{lf}}(\mathbf{T}^n\backslash N)$, so the commutative diagram is equivalent to the chain of restriction maps
\begin{align*}
    H_*^{\mathrm{lf}}(\mathbf{T}^n)\longrightarrow H_*^{\mathrm{lf}}(Y)\longrightarrow H_*^{\mathrm{lf}}(\mathbf{T}^n\backslash N)
\end{align*}
induced by the chain of inclusions $\mathbf{T}^n\backslash N\longrightarrow Y\longrightarrow\mathbf{T}^n$.}
\[
\xymatrix{
H_*(\mathbf{T}^n) \ar[rr]^{i_*} \ar[dr]_{j_{\#}} & & H_*(\mathbf{T}^n, N) \\
&  H_*^{\mathrm{lf}}(Y) \ar[ur]_{i_N} &
}
\]

\textbf{Claim: } $i_*\tau_0 = i_N\circ j_{\#}\tau_0 = i_N(\tau_Y)\in H_1(\mathbf{T}^n,N)$ is a free element.

To see this, let's consider the following exact sequence for relative homology
    \begin{align*}
        H_1(N)\stackrel{\eta_*}{\longrightarrow} H_1(\mathbf{T}^n)\stackrel{i_*}{\longrightarrow} H_1(\mathbf{T}^n,N).
    \end{align*}
    If $m\cdot i_*\tau_0 = i_*(m\tau_0)= 0$ for some $m\in \mathbf{Z}\backslash\{0\}$, then there exists $\bar{\tau_0}\in H_1(N)$, such that $m\tau_0 = \eta_*\bar{\tau}_0$. Since any representative of $\bar{\tau}_0$ is supported in $N$, by Lemma \ref{lem: no intersection}, its algebraic intersection with $\mathbf{T}^{n-1}\times\{\theta_0\}$ must be zero. On the other hand, the algebraic intersection of $m\tau_0$ and $\mathbf{T}^{n-1}\times\{\theta_0\}$ equals $m\ne 0$ since $\tau_0$ is represented by the last $\mathbf{S}^1$ factor in $\mathbf{T}^n$, which is a contradiction. This verifies the Claim.

Now we focus ourselves to the space $Z$. Our argument is based on the following commutative diagram for maps
\[
\xymatrix@C=4cm{
Y
& Y\setminus N \ar[l] \\
Z \ar[u]^{f_Z}
& Z\setminus f_Z^{-1}(N) \ar[l]^{\iota} \ar[u]_{f_Z}
}
\]
Here $\iota: Z\backslash f_Z^{-1}(N)\longrightarrow Z$ is the inclusion. Since $Y\backslash N$ is compact, the properness of $f_Z$ implies $Z\backslash f_Z^{-1}(N)$ is also compact.

Let $\hat{\beta}_k = f_Z^*(\beta_k)\in H^1(Z)$ $(k=1,2,\dots,n-1)$ and $$\tau_Z = [Z]\smallfrown \hat{\beta}_1\smallfrown \hat{\beta}_2\smallfrown\dots\smallfrown \hat{\beta}_{n-1}.$$ The naturality of cap product yields $f_{Z*}(\tau_Z) = (\degg f_Z)\cdot \tau_Y$. From the Claim above and the naturality of the restriction map, we know
\begin{align*}
    f_{Z*}(\tau_Z\big|_{Z\backslash f_Z^{-1}(N)}) =  f_{Z*}(\tau_Z)\big|_{Y\backslash N} = (\degg f_Z)\cdot \tau_Y\big|_{Y\backslash N} = (\degg f_Z)\cdot i_N(\tau_Y)\ne 0
\end{align*}
Denote $\mathcal{T} = [Z]\smallfrown\hat{\beta}_1\smallfrown \hat{\beta}_2\smallfrown\dots\smallfrown \hat{\beta}_{n-2}\in H_2^{\mathrm{lf}}(Z)$, it follows that
\begin{equation}\label{eq: A complicated formula}
    \begin{split}
            \iota_{\#}(\mathcal{T})\smallfrown \iota^*\hat{\beta}_{n-1}=&\iota_{\#}(\mathcal{T}\smallfrown \hat{\beta}_{n-1}) = \iota_{\#}(\tau_Z) = \tau_Z\big|_{Z\backslash f_Z^{-1}(N)}\ne 0.
    \end{split}
\end{equation}
in $ H_1(Z,Z\backslash f_Z^{-1}(N))\cong H_1^{\mathrm{lf}}(Z\backslash f_Z^{-1}(N))$.

\ref{eq: A complicated formula} then implies $\mathcal{T}$ is aspherical in the sense of Definition \ref{def: aspherical homology class}. To see this, it suffices to show $\iota_{\#}(\mathcal{T})$ does not lie in the Hurewicz image of $\pi_2(Z,Z\backslash f_Z^{-1}(N))$. This is because for $\Sigma = \mathbf{S}^2$ or $\mathbf{D}^2$, the geometric representative for $\pi_2(Z,Z\backslash f_Z^{-1}(N))$, the cap product given by
    
    \begin{align*}
H_2(\Sigma,\partial\Sigma)\times H^1(\Sigma)\stackrel{\smallfrown}{\longrightarrow} H_1(\Sigma,\partial\Sigma)
    \end{align*}
    is trivial. Thus, a contradiction would follow from \eqref{eq: A complicated formula} if $\iota_{\#}(\mathcal{T})$ is in the Hurewicz image of $\pi_2(Z,Z\backslash f_Z^{-1}(N))$. Therefore, we conclude the proof that $Z$ is an open SYS manifold.
\end{proof}

We are now in a position to prove Theorem \ref{thm:12 dim Geroch}.
\begin{proof}[Proof of Theorem \ref{thm:12 dim Geroch}] 
     We prove a slightly stronger result that every oriented manifold $M^n$ admitting a non-zero degree map $f: M \longrightarrow \mathbf{T}^n (n\le 12)$ admits no PSC metric. If $n\le 11$, this follows from \cite[Theorem 1.2]{CMSW2025} and the dimension reduction argument in \cite{SY79b}. In the following we focus ourselves on the case $n=12$.
    
    Assume by contradiction that $(M,g)$ satisfies $R_g>0$. From \cite[Theorem 1.2]{CMSW2025}, by possibly perturbing $g$, we are able to find an area-minimizing hypersurface in the homology class $[\Sigma_0]\in H_{n-1}(M)$ with singular set $\mathcal{S}$, such that $\dim_{\mathcal{H}}(\mathcal{S})<1$. Here, $\Sigma_0$ is as in the statement of Lemma \ref{lem: degree calculation}. If $\mathcal{S} = \emptyset$, then we can perform a conformal deformation as in \cite{SY79b} to reduce the problem to the case $n\le 11$. Thus, it suffices to focus on the case $\mathcal{S}\ne \emptyset$.  Using Lemma \ref{lem: degree calculation}, we obtain a proper map 
    \begin{align*}
        F|_{\Sigma\backslash\mathcal{S}}:\Sigma\backslash\mathcal{S}\longrightarrow \mathbf{T}^{n-1}\backslash F(\mathcal{S})
    \end{align*}
    of non-zero degree.  
    
    Since $F$ is Lipschitz, we have $\dim_{\mathcal{H}}(F(\mathcal{S}))\le \dim_{\mathcal{H}}(\mathcal{S})<1$. By Lemma \ref{lem: open SYS verification}, $\Sigma\backslash\mathcal{S}$ is an open SYS manifold. Combined with Lemma \ref{open SYS dim 11}, $\Sigma\backslash\mathcal{S}$ admits no PSC metric.

    Because $\mathcal{S}\ne\emptyset$, from \cite[Lemma 5.1]{HSY26} we can find $p\in\Sigma\backslash S$ such that $|A|^2(p)>0$. Then, we choose a smooth function $\varphi\ge 0$ with compact support on $\Sigma\backslash\mathcal{S}$, such that
    \begin{equation}\label{eq: positive spectral scalar curvature}
    \begin{split}
                &\bullet\varphi(p)>0, \\
        &\bullet\frac{1}{2}(R_g+|A|^2)-\frac{n-1}{n-3}\varphi\ge 0\mbox{ on }\Sigma\backslash \mathcal{S},\\
        &\bullet\frac{1}{2}(R_g+|A|^2)(p)-\frac{n-1}{n-3}\varphi(p) > 0.
    \end{split}
    \end{equation}

    Since $\Sigma$ is compact, $-\Delta$ alone does not admit a Green's function, but a strictly positive operator does (this positivity ensures the existence of a Green's function in general, e.g., even for the noncompact parabolic case); inparticular, the operator $-\Delta+\varphi$ does.  As is shown in \cite[Lemma 5.10]{HSY26}, there exists a Green's type function $G$ on $\Sigma\backslash\mathcal{S}$ satisfying
    \begin{align*}
        -\Delta G+\varphi G = 0, \mbox{ on }\Sigma\backslash\mathcal{S},
    \end{align*}
    which is exactly a Green's function outside the support of $\varphi$. Therefore, the argument in Section 2.1 can be adopted here to construct a blow up factor $G_{\mathcal{S}}$ such that $(\Sigma\backslash\mathcal{S}, \hat{g} = G_{\mathcal{S}}^{\frac{4}{n-3}}g|_{\Sigma})$ is complete.

    Using \cite[Lemma 5.11]{HSY26}, we have
    \begin{align}\label{eq: positive spectral scalar curvature 2}
        \int_{\Sigma\backslash\mathcal{S}} (|\hat{\nabla}\phi|^2+ \frac{1}{2}R_{\hat{g}}\phi^2)d\mu_{\hat{g}}\ge \int_{\Sigma\backslash\mathcal{S}} (\frac{1}{2}(R_g+|A|^2)-\frac{n-1}{n-3}\varphi)G_{\mathcal{S}}^{-\frac{4}{n-3}}\phi^2 d\mu_{\hat{g}}
    \end{align}
    for any $\phi\in C_c^\infty(\Sigma\backslash\mathcal{S})$. From \eqref{eq: positive spectral scalar curvature}, the right hand side of \eqref{eq: positive spectral scalar curvature 2} is nonnegative. By \cite{FcS1980}, there exists a function $v\in C^\infty(\Sigma\backslash\mathcal{S})$ satisfying
    \begin{align*}
        -\Delta_{\hat{g}}v+\frac{1}{2}R_{\hat{g}}v = (\frac{1}{2}(R_g+|A|^2)-\frac{n-1}{n-3}\varphi)G_{\mathcal{S}}^{-\frac{4}{n-3}}v
    \end{align*}
    Hence, the warped manifold $((\Sigma\backslash\mathcal{S})\times\mathbf{S}^1,\bar{g} = \hat{g}+v^2ds^2)$ has $R_{\bar{g}}\ge 0$ everywhere and $R_{\bar{g}}(p)>0$. By an $\mathbf{S}^1$-invariant analogue of \cite{Kazdan82}, which we refer to \cite[Lemma 2.8]{HSY26}, there exists an $\mathbf{S}^1$-invariant conformal factor $w$ on $(\Sigma\backslash\mathcal{S})\times\mathbf{S}^1$ satisfying
    \begin{align*}
        -\Delta_{\bar{g}}w+\frac{n-3}{4(n-2)}R_{\bar{g}}w>0
    \end{align*}
    and $0<a\le w\le b<\infty$ for some positive constants $a$ and $b$. Consequently, $w$ can be regarded as a function on $\Sigma\backslash\mathcal{S}$. Hence, $((\Sigma\backslash\mathcal{S})\times\mathbf{S}^1,w^{\frac{4}{n-3}}\bar{g} = w^{\frac{4}{n-3}}\hat{g}+w^{\frac{4}{n-3}}v^2ds^2)$ is PSC, so $(\Sigma\backslash\mathcal{S}, w^{\frac{4}{n-3}}\hat{g})$ is a complete Riemannian manifold with $\mathbf{S}^1$-warped PSC metric. This contradicts Lemma \ref{open SYS dim 11} since $\Sigma\backslash\mathcal{S}$ is an open SYS manifold.
\end{proof}

\bibliographystyle{alpha}

\bibliography{Positive}

\end{document}